\documentclass[12pt]{article}
\usepackage{amssymb,amsmath,amsthm,tikz,multirow,nccrules,mathrsfs}
\usetikzlibrary{calc, arrows, arrows.meta}

\title{Tilings of Closed Surfaces by Congruent Most Curvilinear Polygons}
\author{Hao Fang Sun, Min Yan \\
Hong Kong University of Science and Technology}

\usepackage[hidelinks, hyperindex]{hyperref}

\newcommand{\mc}{\mathcal}
\newcommand{\bb}{\mathbb}

\newtheorem{theorem}{Theorem}
\newtheorem{lemma}[theorem]{Lemma}
\newtheorem{corollary}[theorem]{Corollary}
\newtheorem{proposition}[theorem]{Proposition}
\newtheorem*{theorem*}{Theorem}

\theoremstyle{definition}
\newtheorem*{definition*}{Definition}
\newtheorem*{case*}{Case}
\newtheorem*{subcase*}{Subcase}

\theoremstyle{remark}

\numberwithin{equation}{section}

\begin{document}

\maketitle

\begin{abstract}
We study tilings of closed surfaces by congruent most
curvilinear polygons. By relating such tilings to branched
coverings determined by planar diagrams and vertex types,
we establish necessary and sufficient conditions for a
given closed topological surface to admit a metric of
constant curvature supporting a tiling with prescribed
combinatorial data and number of tiles.
We also characterize minimal tilings, which do not descend
through nontrivial unbranched coverings to tilings with
fewer tiles, in terms of common block decompositions
of the associated permutations. Explicit permutation
constructions yield infinite families of minimal tilings
with prescribed vertex types, even when the planar diagram and vertex types
are fixed. In the hyperbolic setting, these results
also give finite-index subgroups of the universal
tiling symmetry groups that are maximal among
torsion-free subgroups, with the index equal to
the number of tiles.
\end{abstract}

\section{Introduction}

Unless otherwise stated, all surfaces considered in this paper
are connected, compact, and without boundary. When discussing
\emph{tilings}, we equip the surface with a metric of constant
Gaussian curvature, with respect to which isometry and congruence
are defined. A polygon has corners and sides, and its interior
is homeomorphic to an open disk. Its sides need not be geodesic;
we call such a polygon \emph{curvilinear}.

We study tilings of closed surfaces by congruent
\emph{most curvilinear polygons}. Our main questions concern
the existence of tilings with prescribed combinatorial data
and the existence of \emph{minimal tilings}, which do not
descend through a nontrivial unbranched covering to tilings
with fewer tiles. The combinatorial data are encoded by a
\emph{planar diagram} and prescribed \emph{vertex types}.
By relating these data to branched coverings and permutation
representations, we obtain necessary and sufficient conditions
for existence, a criterion for minimality, and infinite
families of minimal tilings.

Section~2 introduces most curvilinear polygons and the planar
diagrams and vertex types used to describe their tilings.
These notions record the side identifications and the
arrangements of corners at vertices, providing the
combinatorial framework for the rest of the paper.

Section~3 develops the covering interpretation of tilings.
We define descent and minimality, and study the tiling
induced on the universal cover. For the tilings under
consideration, the symmetry group of the universal tiling
acts simply transitively on its tiles. We relate surface
tilings to branched coverings of the surface associated
with the planar diagram
(Lemma~\ref{lem:branched-cover}), and reduce even-sided
planar diagrams to standard forms by cut-and-paste
operations (Lemma~\ref{standard}). This leads to explicit
presentations of the universal tiling symmetry groups
(Theorem~\ref{standard_symmetry}).

Section~4 establishes our main existence theorem
(Theorem~\ref{case12}). Given a closed topological surface,
a planar diagram, prescribed vertex types, and a number
of tiles, the theorem gives necessary and sufficient
conditions for the surface to admit a metric of constant
curvature supporting the required tiling. The conditions
involve divisibility, the parity of the total branching,
Euler characteristic, and orientability. The proof uses
the correspondence with branched coverings and
realizability results for uniform branch data. The metric
is chosen as part of the construction; the theorem does
not require a metric to be fixed in advance.

Section~5 treats polygons with an odd number of sides.
We obtain standard forms for their planar diagrams and
reduce the existence problem to the even-sided case
by replacing the rotationally symmetric side with a
suitable pair of sides.

Section~6 studies minimal tilings by congruent even-sided
most curvilinear polygons. We first characterize minimality
in terms of common block decompositions of the permutations
associated with the branched covering
(Theorem~\ref{minimal criterion}). We then use explicit
permutation constructions to establish three families
of existence results.

More precisely, suppose that the surface $S_D$ associated
with the planar diagram satisfies $\chi(S_D)\le0$.
Let $d'$ be the least positive integer satisfying the
divisibility and parity conditions for the prescribed
vertex types. If the diagram has at least two vertices,
minimal tilings with $sd'$ tiles exist for every sufficiently
large prime $s$ (Theorem~\ref{theorem1}). If the diagram has
one vertex with prescribed type $V^k$, such minimal tilings
exist for every integer $s\ge1$ when $k\ge3$
(Theorem~\ref{onevertex}), and for every $s=2^m$ with
$m\ge0$ when $k=2$ (Theorem~\ref{22}). In particular,
even with the planar diagram and vertex types fixed,
minimal tilings can have arbitrarily many tiles.

Finally, Section~\ref{sec:fuchsian-groups} gives a
group-theoretic interpretation of these results.
The symmetry groups in Theorem~\ref{standard_symmetry}
are Fuchsian groups in the orientable hyperbolic case,
and NEC groups in the nonorientable hyperbolic case.
Surface tilings correspond to finite-index torsion-free
subgroups, with the index equal to the number of tiles.
Minimal tilings correspond to subgroups that are maximal
among torsion-free subgroups. Thus our existence results
also yield subgroups of prescribed indices and infinitely
many conjugacy classes of maximal torsion-free subgroups.

One motivation for studying minimal tilings is to approach
the classification of surface tilings through those that
cannot descend further. Every tiling descends to a minimal
one after finitely many nontrivial descents, since each
descent decreases its number of tiles. Our results show,
however, that this reduction does not lead to a finite
list of minimal tilings, even when the planar diagram
and vertex types are fixed. A classification must therefore
account for these infinite families.

\section{Curvilinear Tiles}

Let $P$ be the curvilinear polygon in a side-to-side tiling, called {\em prototile}. The prototile has {\em sides} and {\em corners}. They are usually called {\em edges} and {\em vertices}, but actually carry extra meaning. A side $s$ of $P$ is not just a curve in the boundary of $P$, but also includes the side of the curve that is inside $P$. The particular side of the curve is illustrated by the shaded region on the left of Figure \ref{general1}. A corner is the meeting place of two compatible sides. 

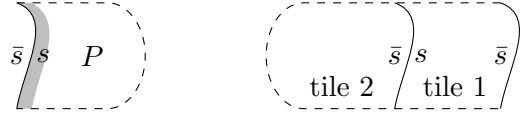
\begin{figure}[htp]
\centering
\begin{tikzpicture}[>=latex,scale=1]


\fill[gray!50]
	(0,0) to[out=80, in=-20] (0,1.4) -- (0.2,1.4) to[out=-30, in=80] (0.2,0);

\draw
	(0,0) to[out=80, in=-20] (0,1.4);

\draw[dashed]
	(0,1.4) -- (1.2,1.4) arc (90:-90:0.5 and 0.7) -- (0,0);

\node at (0.35,0.7) {\small $s$};
\node at (0,0.7) {\small $\bar{s}$};
\node at (1,0.7) {\small $P$};


\begin{scope}[xshift=5cm]

\foreach \a in {0,1}
\draw[xshift=1.4*\a cm]
	(0,0) to[out=80, in=-20] (0,1.4);

\draw[dashed]
	(1.4,1.4) -- (-1.2,1.4) arc (90:270:0.5 and 0.7) -- (1.4,0);
	
\node at (0.35,0.7) {\small $s$};
\node at (0,0.7) {\small $\bar{s}$};
\node at (1.4,0.7) {\small $\bar{s}$};
\node at (0.8,0.3) {\small tile 1};
\node at (-0.7,0.3) {\small tile 2};

\end{scope}

\end{tikzpicture}
\caption{Curvilinear side.}
\label{general1}
\end{figure}

We call the curve part of a side $s$ an {\em edge}. In other words, an edge $e$ ignores the shaded region in a side, and has two sides. If one side is $s$, then we denote the other side by $\bar{s}$. In a {\em side-to-side} tiling, the two sides $s$ and $\bar{s}$ belong to two tiles. Since we assume all tiles are congruent, both $s$ and $\bar{s}$ are sides of $P$. See the right of Figure \ref{general1}.

If $s$ and $\bar{s}$ are distinct, in the sense that there is no self-isometry of the edge that sends one side to the other side, then $s$ and $\bar{s}$ must be different sides of $P$. If $s$ and $\bar{s}$ are the same, in the sense there is such an isometry, then we write $s=\bar{s}$.  

Therefore the symmetry of the edge $e$ characterises whether $s$ and $\bar{s}$ are distinct. There are three nontrivial isometries that can be applied to an edge $e$, such that the set of two end points is preserved: the horizontal flip $e^h$, the vertical flip $e^v$, and the rotation $e^r$, given by Figure \ref{transform}. The horizontal flip can be regarded as the change of direction. Therefore we also write $e^h$ as $e^{-1}$. Then we may assign distinct direction to $e$, i.e., $e\ne e^{-1}$, if and only if $e\ne e^h$. We also write $e=e^{-1}$ in case $e=e^h$. We may apply the same isometries to a side $s$ of $e$, and write $s^h$ as $s^{-1}$. We also write $s=s^{-1}$ if and only if $s=s^h$.

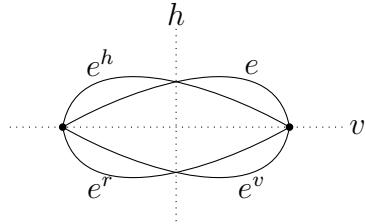
\begin{figure}[htp]
\centering
\begin{tikzpicture}[>=latex,scale=1]

\foreach \a in {1,-1}
\foreach \b in {1,-1}
\draw[xscale=\a,yscale=\b]
	(-1.5,0) to[out=30, in=100] (1.5,0);

\draw[dotted]
	(-2.2,0) -- (2.2,0)
	(0,1.3) -- (0,-1.3);

\fill
	(1.5,0) circle (0.05)
	(-1.5,0) circle (0.05);

\node at (1,0.8) {$e$};
\node at (1,-0.8) {$e^v$};
\node at (-1,0.85) {$e^h$};
\node at (-1,-0.8) {$e^r$};

\node at (2.4,0) {$v$};
\node at (0,1.5) {$h$};

\end{tikzpicture}
\caption{Isometric transformations of a curvilinear edge.}
\label{transform}
\end{figure}

There are four types of edges. Correspondingly, there are four types of sides.
\begin{enumerate}
\item general: $e,e^v,e^h,e^r$ are all distinct. We have $s\ne\bar{s}$, and $s\ne s^{-1}$.
\item $h$-symmetric: $e=e^h\ne e^r$. We have $s\ne\bar{s}$, and $s=s^{-1}$.
\item $r$-symmetric: $e=e^r\ne e^h$. We have $s=\bar{s}$, and $s\ne s^{-1}$.
\item straight: $e=e^v$. This implies $e,e^v,e^h,e^r$ are the same. We have $s=\bar{s}$, and $s=s^{-1}$.
\end{enumerate}

Figure \ref{type} gives examples of the four types of edges, and their schematic drawings. We denote the corresponding four types of sides by $g$ (general), $h$ ($h$-symmetric), $r$ ($r$-symmetric), $a$ (straight arc). We use $g^{-1},r^{-1}$ because they are the sides satisfying $s\ne s^h$, and we do not use $h^{-1},a^{-1}$. We also use $\bar{g}$ and $\bar{h}$ because they are the sides satisfying $s\ne \bar{s}$, and we do not use $\bar{r},\bar{a}$.

\begin{figure}[htp]
\centering
\begin{tikzpicture}[>=latex,scale=1]

\draw
	(-1.5,0) -- (1.5,0)
	(-1.5,0) to[out=80, in=150] (1.5,0)
	(-1.5,0) to[out=50, in=130] (1.5,0)
	(-1.5,0) to[out=-40, in=140] (1.5,0)
	;

\fill
	(1.5,0) circle (0.05)
	(-1.5,0) circle (0.05);

\node at (-1,0.9) {1};
\node at (0.5,0.9) {2};
\node at (-0.5,-0.4) {3};
\node at (0.8,-0.2) {4};


\begin{scope}[xshift=2.5cm]

\draw
	(0,0) -- (0.6,0) -- (0.6,0.1) -- (0.8,0) -- (1.4,0);

\draw[xshift=2cm]
	(0,0) -- (0.5,0) -- (0.7,0.1) -- (0.9,0) -- (1.4,0);

\draw[xshift=4cm]
	(0,0) -- (0.5,0) -- (0.6,-0.1) -- (0.8,0.1) -- (0.9,0) -- (1.4,0);
	
\draw
	(6,0) -- ++(1.4,0);

\node at (0.7,0.4) {$1\; g$};
\node at (2.7,0.4) {$2\; h$};
\node at (4.7,0.4) {$3\; r$};
\node at (6.7,0.4) {$4\; a$};

\end{scope}	

\end{tikzpicture}
\caption{Four types of edges and their schematic drawings.}
\label{type}
\end{figure}
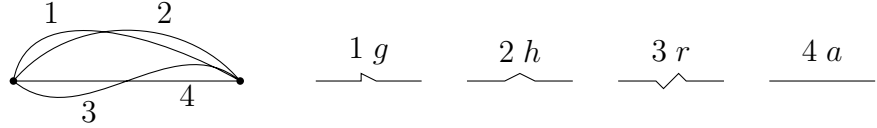

In a tiling of a compact surface, the number of tiles is finite. In a side-to-side tiling by congruent polygons, each edge $e(s)$ means both sides $s$ and $\bar{s}$ appear once in the tiling. If $s\ne\bar{s}$, then this implies the total numbers of the sides $s$ and $\bar{s}$ in the tiling are the same. 

\begin{lemma}\label{parity}
In a side-to-side tiling by congruent curvilinear polygons, $g$ and $\bar{g}$ appear the same number of times in the prototile, and $h$ and $\bar{h}$ appear the same number of times in the prototile. 
\end{lemma}

Consider the extreme case of the {\em most curvilinear} polygon. Such a polygon has the maximal number of sides of the general type. Motivated by Lemma \ref{parity}, we introduce the following definition.\\

\noindent\textbf{Definition.}
A curvilinear polygon \(P\) is called \emph{most curvilinear}
if its sides satisfy one of the following conditions:
\begin{itemize}
    \item If \(P\) has \(2m\) sides, they consist of \(m\)
    distinct pairs \(g_i,\bar g_i\), \(i=1,\ldots,m\),
    of general type.

    \item If \(P\) has \(2m+1\) sides, they consist of \(m\)
    distinct pairs \(g_i,\bar g_i\), \(i=1,\ldots,m\),
    of general type, together with one side of
    \(r\)-symmetric type.
\end{itemize}

The first case is equivalent to a {\em planar diagram} $D$. An example is the left of Figure \ref{oct1}, with the standard notation being a circularly ordered sequence $g_1g_2\bar{g}_1g_3g_4\bar{g}_2\bar{g}_3^{-1}\bar{g}_4^{-1}$. In the notation, each pair $g$ and $\bar{g}$ appears either in {\em opposing} way $\cdots g\cdots \bar{g}^{-1}\cdots$ or in {\em twisted} way $\cdots g\cdots \bar{g}\cdots$, with the superscript $-1$ meaning the clockwise direction, as opposed to the default counterclockwise direction. 

We simplify the presentation of tiling pictures by redrawing and relabelling the planar diagram $D$ as the right of Figure \ref{oct1}. The alphabets $g_i$ and $\bar{g}_i$ in the planar diagram become two $a_i$, one inside $P$ and one outside $P$. The planar diagram becomes $a_1a_2a_1a_3a_4a_2^{-1}a_3^{-1}a_4$. When copies of the prototile are glued together to form a tiling, the labels $a_i$, including the sides where $a_i$ are located, and the arrows must match. 

\begin{figure}[htp]
\centering
\begin{tikzpicture}[>=latex,scale=1]

\begin{scope}[xshift=-3.5 cm]

\foreach \a/\x/\y in 
	{0/-1/1, 1/1/1, 2/1/1, 3/-1/1/, 4/1/1, 5/1/1, 6/-1/-1, 7/-1/-1 } 
\draw[rotate=45*\a, xshift=1.2071 cm, xscale=\x, yscale=\y]
	(0,-0.5) -- (0,-0.1) -- (0.1,-0.1) -- (0,0.1) -- (0,0.5);

\foreach \u/\a in 
	{1/0, 2/1, 3/3, 4/4}
\node at (45+45*\a:1) {\scriptsize $g_{\u}$};

\foreach \u/\a in 
	{1/2, 2/5, 3/6, 4/7}
\node at (45+45*\a:0.95) {\scriptsize $\bar{g}_{\u}$};

\end{scope}

\foreach \a in {0,...,7}
\draw
	(-22.5+45*\a:1.3066) -- (22.5+45*\a:1.3066);

\foreach \a in {0,...,7}
\node[red] at (22.5+45*\a:1.1) {\scriptsize \a};

\foreach \u/\a in 
	{1/0, 2/1, 3/3, 4/4}
\node at (45+45*\a:1.4) {\scriptsize $a_{\u}$};

\foreach \u/\a in 
	{1/2, 2/5, 3/6, 4/7}
\node at (45+45*\a:1) {\scriptsize $a_{\u}$};

\foreach \a in {0,1,2,3,4,7}
\draw[rotate=45+45*\a, ->]
	(1.2071,0) -- ++(0,0.1);

\foreach \a in {5,6}
\draw[rotate=45+45*\a, ->]
	(1.2071,0) -- ++(0,-0.1);

\foreach \a in {0,2,5}
\fill
	(22.5+45*\a:1.3066) circle (0.05);

\foreach \a in {1,3,4,6,7}
\filldraw[fill=white]
	(22.5+45*\a:1.3066) circle (0.05);
	
\end{tikzpicture}
\caption{Octagonal curvilinear prototile $121342^{-1}3^{-1}4$.}
\label{oct1}
\end{figure}
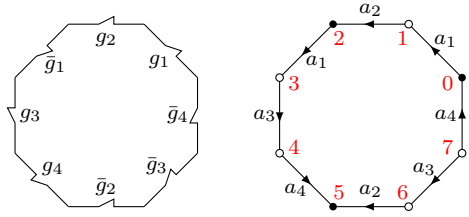

We also label the corners of the prototile $P$ by red numbers ${\color{red} i}\in {\bb Z}_8$, in circular way. The labels assign an orientation of $P$. These labels become the labels of the corners of tiles and also provide the orientations of tiles, by the congruence between tiles and $P$. 

On the left of Figure \ref{oct2}, we have a center tile $t$ congruent to the prototile in Figure \ref{oct1}. Then we have eight tiles sharing edges with $t$. We denote the tile sharing the edge connecting the corners {\color{red} $i$} and {\color{red} $i+1$} of $t$ by $t_i$. The picture shows three such tiles $t_0,t_3,t_6$, and the orientations of the tiles. The matching of sides implies that the eight neighboring tiles $t_i$ are glued to the center tile $t$ in unique ways. 

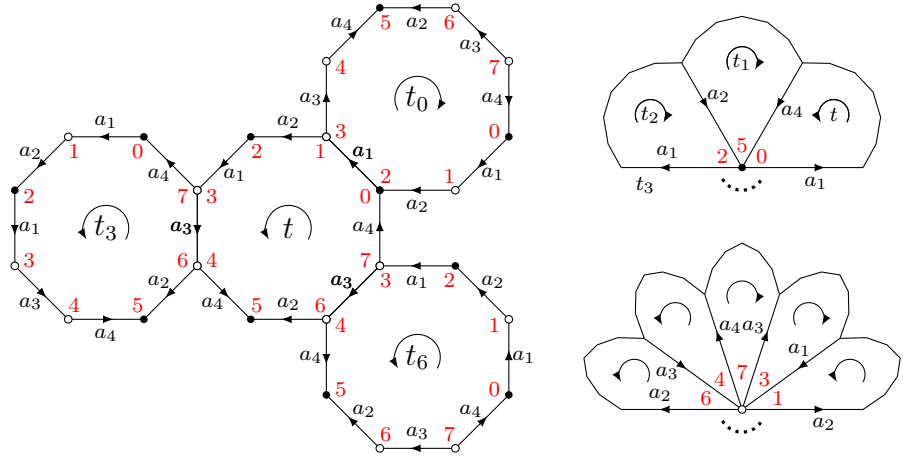
\begin{figure}[htp]
\centering
\begin{tikzpicture}[>=latex,scale=1]

\foreach \u/\v/\w/\x in 
	{0/1/0/0, 4/-1/1/1, 1/1/4/1, -1/1/-1/1}
{
\begin{scope}[shift={(45*\w:1.2071*2*\x)}, rotate=45*\u,  xscale=\v]

\foreach \a in {0,...,7}
\draw
	(-22.5+45*\a:1.3066) -- (22.5+45*\a:1.3066);

\foreach \a in {0,...,7}
\node[red] at (22.5+45*\a:1.1) {\scriptsize \a};

\foreach \u/\a in 
	{1/0, 2/1, 3/3, 4/4}
\node at (45+45*\a:1.41) {\scriptsize $a_{\u}$};

\foreach \u/\a in 
	{1/2, 2/5, 3/6, 4/7}
\node at (45+45*\a:1) {\scriptsize $a_{\u}$};

\foreach \a in {0,1,2,3,4,7}
\draw[rotate=45+45*\a, ->]
	(1.2071,0) -- ++(0,0.1);

\foreach \a in {5,6}
\draw[rotate=45+45*\a, ->]
	(1.2071,0) -- ++(0,-0.1);

\foreach \a in {0,2,5}
\fill
	(22.5+45*\a:1.3066) circle (0.05);

\foreach \a in {1,3,4,6,7}
\filldraw[fill=white]
	(22.5+45*\a:1.3066) circle (0.05);
	
\end{scope}
}

\node at (0,0) {$t$};

\draw[->]
	(-30:0.3) arc (-30:210:0.3);
	
\foreach \a/\b in 
	{0/-1, 3/1, 6/1}
{ 
\draw[->, shift={(45+45*\a:1.2071*2)}, xscale=\b]
	(-30:0.3) arc (-30:210:0.3);
\node at (45+45*\a:1.2071*2) {$t_{\a}$};
}

\begin{scope}[shift={(6cm,0.8cm)}]

\foreach \a in {0,...,3}
\draw
	(0,0) -- (60*\a:1.6);

\foreach \a in {0,1,2}
\draw[rotate=60*\a]	
	(0:1.6) -- (10:1.8) -- (20:1.95) -- (30:2) -- (40:1.95) -- (50:1.8) -- (60:1.6);
	
\fill (0,0) circle (0.05);

\draw[dotted, very thick]
	(210:0.3) arc (210:330:0.3);

\foreach \a/\x in {0/0, 1/5, 2/2}
\node[red] at (30+60*\a:0.3) {\scriptsize \x};

\foreach \a/\x/\u in {0/1/-1, 1/4/-1, 2/2/-1, 3/1/-1}
\node at (12*\u+60*\a:1) {\scriptsize $a_{\x}$};

\foreach \a/\b in {0/1, 1/-1, 2/-1, 3/1}
\draw[->]
	(60*\a:1) -- ++(60*\a:0.1*\b);

\foreach \a/\b in 
	{0/1, 1/-1, 2/-1} 
\draw[->, shift={(30+60*\a:1.4)}, xscale=\b]
	(-30:0.2) arc (-30:210:0.2);

\foreach \a/\b in 
	{0/1, 1/-1, 2/-1} 
\draw[->, shift={(30+60*\a:1.4)}, xscale=\b]
	(-30:0.2) arc (-30:210:0.2);

\node at (30:1.4) {\scriptsize $t$};
\node at (90:1.4) {\scriptsize $t_1$};
\node at (150:1.4) {\scriptsize $t_2$};
\node at (190:1.3) {\scriptsize $t_3$};
	
\end{scope}

\begin{scope}[shift={(6cm,-2.4cm)}]

\foreach \a in {0,...,5}
\draw
	(0,0) -- (36*\a:1.6);

\foreach \a in {0,...,4}
\draw[rotate=36*\a]	
	(0:1.6) -- (6:1.9) -- (12:2.1) -- (18:2.2) -- (24:2.1) -- (30:1.9) -- (36:1.6);

\filldraw[fill=white] (0,0) circle (0.05);

\draw[dotted, very thick]
	(210:0.3) arc (210:330:0.3);
	
\foreach \a/\x in {0/1, 1/3, 2/7, 3/4, 4/6}
\node[red] at (18+36*\a:0.5) {\scriptsize \x};

\foreach \a/\x/\u in {0/2/-1, 1/1/1, 2/3/1, 3/4/-1, 4/3/1, 5/2/-1}
\node at (10*\u+36*\a:1.1) {\scriptsize $a_{\x}$};

\foreach \a/\b in {0/1, 1/-1, 2/1, 3/1, 4/-1, 5/1}
\draw[->]
	(36*\a:1) -- ++(36*\a:0.1*\b);

\foreach \a/\b in 
	{0/1, 1/-1, 2/-1, 3/1, 4/1} 
\draw[->, shift={(18+36*\a:1.5)}, xscale=\b]
	(-30:0.2) arc (-30:210:0.2);
	
\end{scope}

\end{tikzpicture}
\caption{Unique gluing and unique vertex.}
\label{oct2}
\end{figure}

On the upper right of Figure \ref{oct2}, we start with the corner ${\color{red} 0_+}$ of a tile at the vertex $\bullet$, with the subscript $+$ indicating the tile is counterclockwise oriented. Then by the unique ways the eight neighboring tiles are glued to this tile, we determine the adjacent tile with the corner ${\color{red} 5_-}$ at $\bullet$, with the subscript $-$ indicating the tile is clockwise oriented. By the same reason, we further determine the next adjacent tile with the corner ${\color{red} 2_-}$ at $\bullet$. Then we arrive at an edge $1$ that has the same relation to $\bullet$ as the starting edge $1$. This means the further corners we obtain repeats the ${\color{red} 2_-5_-0_+}$ cycle. Here the notation is in the reverse order ${\color{red} 2_-}\leftarrow {\color{red} 5_-}\leftarrow {\color{red} 0_+}$, for future convenience. Therefore the vertex $\bullet={\color{red} \cdots 2_-5_-0_+2_-5_-0_+2_-5_-0_+}=({\color{red} 2_-5_-0_+})^k$ for an integer $k$. 

In fact, if we start with corner {\color{red} 5} or {\color{red} 2} at a vertex, we still conclude the same vertex $({\color{red} 0_+2_-5_-})^k=({\color{red} 5_-0_+2_-})^k=({\color{red} 2_-5_-0_+})^k$. Here the notation for vertices are circularly ordered. Since the sum of the angle values of all corners at the vertex is $2\pi$, we conclude the multiplicity $k$ is the same at all such vertices.

By the same reason, we find another vertex $\circ=({\color{red} 6_+4_+7_-3_-1_+})^{k'}$. Again the multiplicity $k'$ is the same at all such vertices.

In general, for even $n=2m$, a most curvilinear prototile $n$-gon $P$ is a planar diagram $D$. If we glue the edges of $P$ according to $D$, then we get a single tile tiling ${\mc T}_D$ of a surface $S_D$. This tiling has several vertices $V_1,V_2,\dots,V_v$. For example, the planar diagram in Figure \ref{oct1} determines a single tile tiling of a surface, with two vertices $V_1={\color{red} 2_-5_-0_+}$ and $V_2={\color{red} 6_+4_+7_-3_-1_+}$. In particular, the surface has Euler number $2-4+1=-1$, which implies the surface $S_D=3{\bb P}^2$. 

The discussion for the octagon example shows that, if ${\mc T}$ is a tiling of a surface $S$ with $P$ as the (combinatorial) prototype, then the vertices of ${\mc T}$ are $V_1^{k_1},V_2^{k_2},\dots,V_v^{k_v}$, for some constant natural numbers $k_1,k_2,\dots,k_v$. Here the prototile is combinatorial because it satisfies the condition that the sum of angles in $V_i$ is $\frac{2\pi}{k_i}$. However, the vertices in the single tile tiling of $S_D$ have $k_i=1$, and the sum of angles in $V_i$ is $2\pi$. Therefore the prototiles in the tiling of $S$ may not be congruent to the prototile in the single tile tiling of $S_D$. They only have the same planar diagram $D$.

\section{Coverings of Tilings}
Let \(\mc T\) be a tiling of a closed, connected surface \(S\) equipped with a metric of constant curvature, and let \(q:\widehat S\to S\) be an unbranched covering map. We equip \(\widehat S\) with the pullback metric, which has the same constant curvature and makes \(q\) a local isometry. The tiling \(\mc T\) then lifts to a tiling \( \widehat{\mc T}\) of \(\widehat S\), with each lifted tile isometric to the corresponding tile of \(\mc T\). Thus a surface tiling gives rise to further surface tilings through unbranched coverings. From this viewpoint, there are three types of surface tilings:

\begin{itemize}
\item $S={\bb S}^2$ (sphere) or ${\bb P}^2$ (projective plane): The sphere is the double cover of the projective plane, and is the universal cover $\tilde{S}$.
\item $S={\bb T}^2$ (torus) or $2{\bb P}^2$ (Klein bottle): The torus covers torus and Klein bottle. The Klein bottle covers Klein bottle. The universal cover $\tilde{S}$ is the plane ${\bb R}^2$.
\item $S=g{\bb T}^2$ ($g\ge 2$) or $g{\bb P}^2$ ($g\ge 3$): The oriented surface $g{\bb T}^2$ covers $g'{\bb T}^2$ and $g'{\bb P}^2$. The nonorientable surface $g{\bb P}^2$ covers $g'{\bb P}^2$. The relation between $g$ and $g'$ is multiples of the Euler characteristic. The universal cover $\tilde{S}$ is the Poincar\'e disk ${\bb H}^2$.
\end{itemize}

The classification of surface tilings are therefore divided into three types: spherical, planar, and hyperbolic. In \cite{article5}, we classified edge-to-edge tilings of the sphere by curvilinear triangles. We will see that, tilings by congruent most curvilinear polygons are almost always hyperbolic.\\ 

Let \(\mc T\) be a tiling of a closed constant-curvature surface \(S\) by congruent most curvilinear polygons. We say that \(\mc T\) \emph{descends} to a tiling \(\mc T'\) of a closed constant-curvature surface \(S'\) if there exists an unbranched covering map
\[
q:S\longrightarrow S'
\]
that is a local isometry and maps each tile of \(\mc T\) isometrically onto a tile of \(\mc T'\), preserving the corresponding sides and corners.

The tiling $\mc T$ is called \emph{minimal} if, whenever $\mc T$ descends to
a tiling $\mc T'$ through an unbranched covering
\[
q:S\longrightarrow S',
\]
the covering $q$ has degree $1$.

Equivalently, $\mc T$ is minimal if it does not descend, through any
nontrivial unbranched covering, to a tiling with fewer tiles.

Due to the covering of tilings, the classification of surface tilings is actually the classification of  minimal tilings. The concept of minimal tiling is especially important for hyperbolic surface tilings, because a hyperbolic surface has infinitely many coverings with distinct genus. 

Tilings of surfaces are the quotients of tilings of the universal cover $\tilde{S}={\bb S}^2,{\bb R}^2,{\bb H}^2$. From this viewpoint, for a given (curvilinear) prototile $P$, we consider tilings $\tilde{\mc T}$ of the universal cover $\tilde{S}$ by congruent copies of $P$, such that the symmetry group $\Gamma$ of $\tilde{\mc T}$ has finitely many orbits. Then we find subgroups $G$ of $\Gamma$ of finite index, such that $G$ has no fixed points. Then the quotient $\tilde{\mc T}/G={\mc T}$ is a tiling of the quotient surface $\tilde{S}/G=S$.

Now we specialize to tilings ${\mc T}$ by tiles congruent to a most curvilinear polygon $P$ associated to a planar diagram $D$. We fix one tile $t$ in the universal cover $\tilde{\mc T}$. Let $t'$ be any other tile in $\tilde{\mc T}$. Since both $t$ and $t'$ are isometric to $P$, and $P$ has no self symmetry, there is a unique isometry $\gamma$ of $\tilde{S}$ taking $t$ to $t'$. 

Any vertex of $t$ is $V_i^{k_i}$. Since vertices are unique, the corresponding vertex of $t'$ is also $V_i^{k_i}$. The uniqueness of the vertex implies the isometry $\gamma$ between $t$ and $t'$ is actually an isometry between all the tiles at the vertex. Since this is true for all the vertices of $t$, we know $\gamma$ is an isometry between the first layer neighborhoods $N_1(t)$ and $N_1(t')$. Here $N_0(t)=t$, and the $n$-th layer neighborhood $N_n(t)$ means all the tiles sharing common edges or vertices with $N_{n-1}(t)$. Applying the same argument to tiles in $N_1(t)$ in place of $t$, we find $\gamma$ is an isometry between $N_2(t)$ and $N_2(t')$. Continuing the argument, we find $\gamma$ is an isometry of the universal tiling $\tilde{\mc T}$ to itself. In other words, we have $\gamma\in \Gamma$.

The argument above shows that the map
\[
\gamma\in \Gamma \to \gamma t\in \tilde{\mc T}
\]
is a one-to-one correspondence. Therefore $\tilde{\mc T}/\Gamma$ is the single tile tiling of the surface $S_D$ determined by the planar diagram. 

\begin{lemma}\label{lem:branched-cover}
Let $P$ be a combinatorial most curvilinear $2m$-gon with planar diagram
$D$, and let $\mc T_D$ be the associated single tile tiling
of $S_D$, with vertices $V_1,\ldots,V_v$.

\begin{enumerate}
\item
A tiling $\mc T$ of a surface $S$ by congruent most curvilinear
polygons with combinatorial prototile $P$ and vertex
types $V_1^{k_1},\ldots,V_v^{k_v}$ induces a canonical
branched covering $S\to S_D$, branched at the
vertices of $\mc T_D$, with local degree $k_i$ at every
point over $V_i$.

\item
Conversely, given a finite-degree branched covering
$f:S\to S_D$ with these local degrees and no other
branching, there is a constant-curvature metric on $S$
for which $f^{-1}(\mc T_D)$ is a tiling by congruent most
curvilinear polygons with combinatorial prototile $P$.
The induced map in part~1 is exactly $f$.
\end{enumerate}

The curvature in part~2 has the sign of
\[
    1-m+\sum_{i=1}^{v}\frac1{k_i}.
\]
\end{lemma}

\begin{proof}
For part~1, the canonical tilewise maps agree along
shared sides and are local homeomorphisms away from
vertices. At a vertex of type $V_i^{k_i}$, the local
model is $z\mapsto z^{k_i}$.

For part~2, regard $S_D$ as the orbifold $\mathcal O$
with cone order $k_i$ at $V_i$.
The local-degree assumption makes $f:S\to\mathcal O$
an orbifold covering with nonsingular source, so
$\mathcal O$ is good.
By Theorem 13.3.6 in \cite{ThurstonGT3M}, it admits a
constant-curvature structure whose curvature has the
sign of
\[
    \chi_{\mathrm{orb}}(\mathcal O)
    =\chi(S_D)-\sum_i(1-k_i^{-1})
    =1-m+\sum_i k_i^{-1}.
\]
Keeping the constant-curvature metric fixed, we slightly perturb the edges of the single-tile decomposition \(\mc T_D\), away from the cone points, so that they are of general type and distinct edge pairs are noncongruent if necessary. This preserves the planar diagram \(D\).  Its pullback is smooth on $S$, including over the cone
points, where the covering is locally
$\mathbb D\to\mathbb D/C_{k_i}$.
Since the open tile of $\mc T_D$ is simply connected and
contains no cone points, each lifted open tile maps
isometrically onto it under $f$.
Thus the lifted tiling consists of congruent tiles
with planar diagram $D$ and vertex types
$V_1^{k_1},\ldots,V_v^{k_v}$.
The canonical tilewise maps are the restrictions of
$f$, so the branched covering induced by this tiling
is the given map $f$.
\end{proof}

An immediate question is the following, which is equivalent to finding suitable branched covers of $S_D$.

\medskip

\noindent {\bf Question.} Suppose a planar diagram $D$ has vertices $V_1,V_2,\dots,V_v$, and $k_1,k_2,\dots,k_v$ are natural numbers. What is the necessary and sufficient condition for a surface to have a tiling by curvilinear polygons congruent to a prototile given by $D$, such that the vertices are  $V_1^{k_1},V_2^{k_2},\dots,V_v^{k_v}$?

\medskip

Let us determine the symmetry group $\Gamma$ for the planar diagram in Figure \ref{oct1}. On the left of Figure \ref{oct2}, we have the initial tile $t$. If $t_j$ shares common edge $a_i$ with $t$, and $a_i$ is {\em outside} $t$, then we denote isometry of $\tilde{S}$ that sends $t$ to $t_j$ simply by $\gamma=a_i$. This means
\[
t_0=a_1t,\quad
t_1=a_2t,\quad
t_3=a_3t,\quad
t_4=a_4t.
\]
The isometry $a_3$ carries eight tiles sharing edges with $t$ to eight tiles sharing edges with $t_3$. Then we see this isometry $a_3$ carries $t_6$ in $N_1(t)$ to $t$ in $N_1(t_3)$. Therefore we conclude $t=a_3t_6$. Then $t_6=a_3^{-1}t$, similarly, we obtain
\[
t_2=a_1^{-1}t,\quad
t_5=a_2^{-1}t,\quad
t_6=a_3^{-1}t,\quad
t_7=a_4^{-1}t.
\]
We see that the eight tiles sharing edges with $t$ are obtained by applying four isometries $a_1,a_2,a_3,a_4$, and their inverses $a_1^{-1},a_2^{-1},a_3^{-1},a_4^{-1}$ to $t$. 

We remark that, if the initial $t$ on the left of Figure \ref{oct2} is replaced by $\gamma t$ for some $\gamma\in \Gamma$, then the eight tiles around $\gamma t$ are obtained by the applying the conjugations by $\gamma$. For example, the three tiles $t_0,t_3,t_6$ become three tiles $\gamma a_1\gamma^{-1}(\gamma t)=\gamma a_1t$, $\gamma a_3\gamma^{-1}(\gamma t)=\gamma a_3t$, $\gamma a_3^{-1}\gamma^{-1}(\gamma t)=\gamma a_3^{-1}t$.

To get isometries that reach all the tiles in $N_1(t)$, we examine the isometries that relate the tiles at the vertices on the upper right of Figure \ref{oct2}. We assume the starting tile is $t$, and denote the subsequent tiles by $t_1,t_2,t_3,\cdots$. By the remark about the conjugation by $\gamma$, we have 
\[
t_1=a_4^{-1}t,\quad
t_2=a_2^{-1}a_4^{-1}t,\quad
t_3=a_1^{-1}a_2^{-1}a_4^{-1}t.
\]
The next tiles are obtained by successively applying $a_4^{-1},a_2^{-1},a_1^{-1}$ in cycles. Then the vertex $\bullet=({\color{red} 2_-5_-0_+})^k$ implies $t=(a_1^{-1}a_2^{-1}a_4^{-1})^kt$. By the one-to-one correspondence $\gamma\in \Gamma \to \gamma t\in \tilde{\mc T}$, this further implies $(a_1^{-1}a_2^{-1}a_4^{-1})^k=id$. 

The labels on the right of Figure \ref{oct1} lead to a dictionary $\delta_D$ on how to convert corners (decorated with signs $\pm$) at vertices to isometries.
\begin{center}
\footnotesize
\begin{tabular}{|c||c|c|c|c|c|c|c|c|}
\hline
sign & {\color{red} 0} 
& {\color{red} 1} 
& {\color{red} 2} 
& {\color{red} 3} 
& {\color{red} 4} 
& {\color{red} 5} 
& {\color{red} 6} 
& {\color{red} 7}  \\
\hline \hline
$+$ & $a_4^{-1}$ & $a_1$ & $a_2$ & $a_1^{-1}$ & $a_3$ & $a_4$ & $a_2^{-1}$ & $a_3^{-1}$  \\
\hline
$-$ & $a_1$ & $a_2$ & $a_1^{-1}$ & $a_3$ & $a_4$ & $a_2^{-1}$ & $a_3^{-1}$ & $a_4^{-1}$ \\
\hline
\end{tabular}
\end{center}

For example,
\[
\delta_D({\color{red} 2_-5_-0_+})=a_1^{-1}a_2^{-1}a_4^{-1}.
\]
Similarly,
\[
\delta_D({\color{red} 6_+4_+7_-3_-1_+})
=a_2^{-1}a_3a_4^{-1}a_3a_1.
\]
The vertex $\circ$ therefore gives the relation
\[
(a_2^{-1}a_3a_4^{-1}a_3a_1)^{k'}=id.
\]

By applying combinations of the isometries $a_1,a_2,a_3,a_4,a_1^{-1},a_2^{-1},a_3^{-1},a_4^{-1}$ to the initial tile $t$, we can get all the tiles in $\tilde{\mc T}$. This means the four isometries generate $\Gamma$. Moreover, the relations for the generators are the equalities obtained at all the vertices. We already obtain the relations $(a_1^{-1}a_2^{-1}a_4^{-1})^k=id$ and $(a_2^{-1}a_3a_4^{-1}a_3a_1)^{k'}=id$ from the vertices of $t$. The relations obtained at the other vertices are the conjugations of these two relations. Therefore these two are the only relations, and we get
\[
\Gamma
=\langle a_1,\; a_2,\; a_3,\; a_4\;  |\;  (a_1^{-1}a_2^{-1}a_4^{-1})^k,\; (a_2^{-1}a_3a_4^{-1}a_3a_1)^{k'}\rangle.
\]

In general, for any most curvilinear $2m$-gon prototile $P$ associated to a planar diagram $D$, let $V_1,V_2,\dots,V_v$ be the vertices of the single tile tiling ${\mc T}_D$. Then we can derive a dictionary $\delta$ that converts the decorated corners ${\color{red} i_{\pm}}$ into isometries $j$ or $j^{-1}$.

\begin{lemma}\label{symmetry}
Let ${\mc T}$ be a tiling of $S$ by congruent copies of $P$, with vertices $V_1^{k_1},V_2^{k_2},\dots,V_v^{k_v}$. Then the universal cover tiling $\tilde{\mc T}$ of $\tilde{S}$ has the symmetry group
\[
\Gamma
=\langle a_1,a_2,\dots,a_m\; | \; 
\delta(V_1)^{k_1},\delta(V_2)^{k_2},\dots,\delta(V_v)^{k_v}\rangle.
\]
\end{lemma}

Any closed and connected surface is one of ${\bb S}^2$, $g{\bb T}^2$, $g{\bb P}^2$. One way to prove this fact is to use cut and paste to modify general planar diagrams to standard planar diagrams for the specific surfaces. The operation can also be applied to most curvilinear tilings. On the left of Figure \ref{cutpaste}, we cut the octagonal prototile $P$ into two pieces along a general curvilinear line $a_{3'}$ connecting the corners {\color{red} 2} and {\color{red} 5}. Then we glue the two pieces together along the edge pair 3 to form a new prototile $P'$. The edge pair $a_3$ becomes a line in the interior of $P'$, and the line $a_{3'}$ becomes an edge pair for $P'$. Due to the rigidity of the tiling ${\mc T}$ with prototile $P$, applying the cut and paste operation to ${\mc T}$ gives a tiling ${\mc T}'$ with prototile $P'$. Two tiles of ${\mc T}'$ are indicated by shaded regions in the picture.

\begin{figure}[htp]
\centering
\begin{tikzpicture}[>=latex,scale=1]

\fill[gray!30]
	(-0.48,1.2071) -- ++(0:0.98) -- ++(-45:1) -- ++(-90:1) -- ++(0:1) -- ++(225:2.414) -- ++(90:1) -- ++(180:0.98)
	(-0.52,1.2071) -- ++(225:0.97) -- ++(135:1) -- ++(180:1) -- ++(225:1) -- ++(-45:2.414) -- ++(45:1) -- ++(-45:0.97);

\foreach \u/\v/\w/\x in 
	{0/1/0/0, 4/-1/1/1, 1/1/4/1, -1/1/-1/1}
{
\begin{scope}[shift={(45*\w:1.2071*2*\x)}, rotate=45*\u,  xscale=\v]

\foreach \a in {0,...,7}
\draw
	(-22.5+45*\a:1.3066) -- (22.5+45*\a:1.3066);

\foreach \a in {0,...,7}
\node[red] at (22.5+45*\a:1.1) {\scriptsize \a};

\foreach \u/\a in 
	{1/0, 2/1, 3/3, 4/4}
\node at (45+45*\a:1.41) {\scriptsize $a_{\u}$};

\foreach \u/\a in 
	{1/2, 2/5, 3/6, 4/7}
\node at (45+45*\a:1) {\scriptsize $a_{\u}$};

\foreach \a in {0,1,2,3,4,7}
\draw[rotate=45+45*\a, ->]
	(1.2071,0) -- ++(0,0.1);

\foreach \a in {5,6}
\draw[rotate=45+45*\a, ->]
	(1.2071,0) -- ++(0,-0.1);

\foreach \a in {0,2,5}
\fill
	(22.5+45*\a:1.3066) circle (0.05);

\foreach \a in {1,3,4,6,7}
\filldraw[fill=white]
	(22.5+45*\a:1.3066) circle (0.05);

\draw[dashed]
	(112.5:1.3066) -- (247.5:1.3066);

\node at (-0.7,0) {\scriptsize $a_{3'}$};
	
\end{scope}
}

\begin{scope}[shift={(5.5cm,1.6cm)}]

\draw
	(1,0.6) -- (1,-0.6) 
	(-1,0.6) -- (-1,-0.6) 
	(0,1) -- (0,-1);

\draw[dashed]
	(1,0.6) -- (0,1) -- (-1,0.6)
	(1,-0.6) -- (0,-1) -- (-1,-0.6);

\foreach \a in {1, -1, 0}
\draw[->]
	(\a,0) -- ++(0,0.1);

\node at (-1.2,0) {\scriptsize $a_i$};
\node at (0.8,0) {\scriptsize $a_i$};
\node at (-0.2,0) {\scriptsize $a_{i'}$};
\node at (1.5,0) {\scriptsize $P$};

\end{scope}

\draw[very thick, ->]
	(5.5,0.3) -- ++(0,-1);

\begin{scope}[shift={(5.5cm,-1.6cm)}]

\draw
	(1,1) -- (1,-1) 
	(-1,1) -- (-1,-1) 
	(0,0.6) -- (0,-0.6);

\draw[dashed]
	(1,1) -- (0,0.6) -- (-1,1)
	(1,-1) -- (0,-0.6) -- (-1,-1);

\foreach \a in {1, -1, 0}
\draw[->]
	(\a,0) -- ++(0,0.1);

\node at (-1.2,0) {\scriptsize $a_{i'}$};
\node at (0.8,0) {\scriptsize $a_{i'}$};
\node at (-0.2,0) {\scriptsize $a_i$};
\node at (1.5,0) {\scriptsize $P'$};

\end{scope}

\end{tikzpicture}
\caption{Cut and paste.}
\label{cutpaste}
\end{figure}
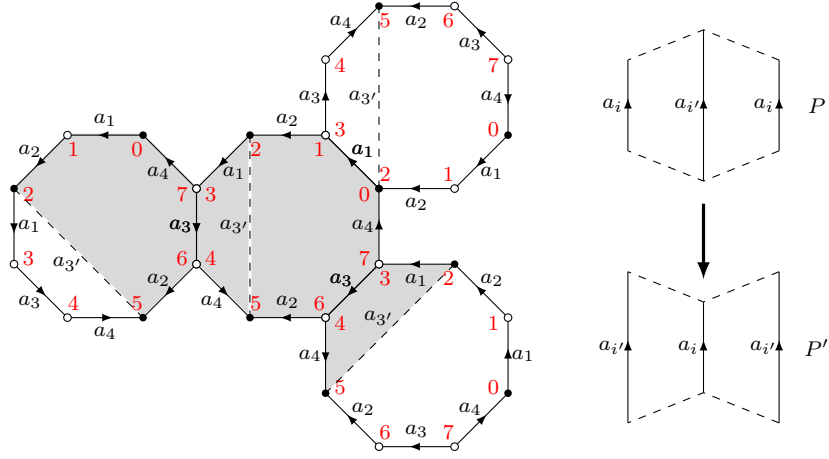

The right of Figure \ref{cutpaste} is the general cut and paste operation. We cut $P$ into two pieces along a general curvilinear line $a_{i'}$ connecting two corners of $P$, so that an edge pair appears on two sides of $a_{i'}$. Then we glue the two pieces along the edge pair $a_i$ to form a new prototile $P'$, with the edge pair $a_i$ replaced by the edge pair $a_{i'}$. Then rigidity implies that a tiling ${\mc T}$ with prototile $P$ becomes a tiling ${\mc T}'$ with prototile $P'$. Since cut and paste can be reversed, i.e., we can get $P$ from $P'$ by using cut and paste, we know that tilings with prototile $P$ are equivalent to tilings with prototile $P'$. 

We remark that we need to reassign circularly cyclic labels to the corners of $P'$. Moreover, the locations of the vertices on the surface remain the same. However, the corners at the vertices are divided and combined. Therefore, the corner combinations at vertices are changed.

The equivalence of the tilings in the following result is achieved by repeated use of cut and paste. 

\begin{lemma}\label{standard}
Any tiling by congruent most curvilinear polygons of even number of sides is equivalent to a tiling by congruent most curvilinear polygons associated to one of the following standard planar diagrams:
\begin{itemize}
\item $a_1b_1a_1^{-1}b_1^{-1}\cdots a_gb_ga_g^{-1}b_g^{-1}c_1c_1^{-1}\cdots c_pc_p^{-1}$.
\item $a_1a_1\cdots a_ga_gc_1c_1^{-1}\cdots c_pc_p^{-1}$.
\end{itemize}
\end{lemma}

\begin{proof}
The proof is a variation of the proof for the classification of surfaces. We say corners of $D$ and vertices of ${\mc T}_D$. The set of corners is divided into a disjoint union of subsets $V_1,V_2,\cdots,V_v$, which are exactly the vertices of the single tile tiling ${\mc T}_D$. We label the corners by the corresponding vertex $V_i$ of ${\mc T}_D$. The number of appearances of a vertex $V_i$ in the planar diagram is the degree of $V_i$. If the degree of $V_i$ is $1$, then the corresponding corner appears in the planar diagram only once. This is equivalent to that it is the middle corner in a segment $cc^{-1}$ of the planar diagram. See $V_1,V_2,V_3$ in Figure \ref{cutpaste1}, which are the middle corners of $c_1c_1^{-1},c_2c_2^{-1},c_3c_3^{-1}$.

The first step is to use cut and paste to reduce the degrees of all but one $V_i$ to become $1$. Suppose that the degrees of $V_1,V_2,\dots,V_{s-1}$ are already $1$. If $V_s$ is only connected by edges to $V_1,V_2,\dots,V_{s-1},V_s$, then starting with one corner $V_s$ and traversing around a planar diagram, we successively get $V_s,V_s,\dots,V_s,V_j$, with $j<s$. Since the degree of $V_j$ is $1$, we know that $V_s$ is connected to $V_j$ through an edge $a$, and the subsequent edge is $a^{-1}$ connecting $V_j$ to $V_s$. Therefore, we get the sequence $V_s,V_s,\dots,V_s,V_j,V_s,\dots,V_s,V_{j'},V_s,\dots$, with $j,j',\cdots <s$. See Figure \ref{cutpaste1}. This implies $V_1,V_2,\dots,V_{s-1},V_s$ are all the vertices, and $V_s$ is the only vertex that can have degree $>1$. This is exactly what we try to achieve.

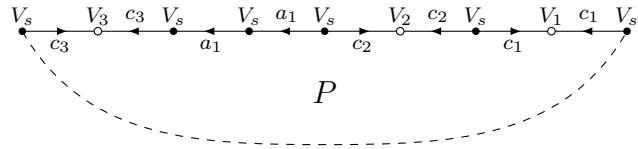
\begin{figure}[htp]
\centering
\begin{tikzpicture}[>=latex,scale=1]

\draw
	(0,0) -- (8,0);

\foreach \a in {0,4,6}
\draw[->]
	(0.5+\a,0) -- ++(0.1,0);

\foreach \a in {1,2,3,5,7}
\draw[->]
	(0.5+\a,0) -- ++(-0.1,0);

\foreach \a in {0,2,3,4,6,8}
\fill
	(\a,0) circle (0.05);

\foreach \a in {1,5,7}
\filldraw[fill=white]
	(\a,0) circle (0.05);

\foreach \a in {0,2,3,4,6,8}
\node at (\a,0.2) {\scriptsize $V_s$};

\draw[dashed]
	(0,0) to[out=-60, in=180] (4,-1.5) to[out=0, in=240] (8,0);

\begin{scope}[font=\scriptsize]

\node at (0.5,-0.2) {$c_3$};
\node at (1.5,0.2) {$c_3$};
\node at (2.5,-0.2) {$a_1$};
\node at (3.5,0.2) {$a_1$};
\node at (4.5,-0.2) {$c_2$};
\node at (5.5,0.2) {$c_2$};
\node at (6.5,-0.2) {$c_1$};
\node at (7.5,0.2) {$c_1$};

\foreach \a/\b in {1/3, 5/2, 7/1}
\node at (\a,0.2) {$V_{\b}$};

\end{scope}

\node at (4, -0.8) {$P$};

\end{tikzpicture}
\caption{$V_s$ connected to $V_1,V_2,V_3$.}
\label{cutpaste1}
\end{figure}

Therefore we may assume $V_s$ is connected to $V_t$ by an edge $a$, with $t>s$. Suppose the $V_s$ end of $a$ is connected to an edge $b$. Then another $b$ appears in the planar diagram as on the left of Figure \ref{cutpaste2}. The dashed part of the planar diagram are combinations of edges. The two pictures correspond to the cases that the $b$ pair is opposing or twisted.

\begin{figure}[htp]
\centering
\begin{tikzpicture}[>=latex,scale=1]

\foreach \y in {0,1}
{
\begin{scope}[yshift=-2.5*\y cm]

\draw
	(0,0.7) -- (1.4,0.7) -- (1.4,-0.7) -- (0,0.7)
	(0,-0.7) -- (-1.4,-0.7);

\draw[dashed]
	(0,0.7) -- (-1.4,-0.7)
	(0,-0.7) -- (1.4,-0.7);

\draw[->]
	(0.7,0.7) -- ++(-0.1,0);
\draw[->]
	(1.4,0) -- ++(0,0.1);
\draw[->]
	(0.7,-0.7) -- ++(0.1,0);
\draw[->]
	(0.7,0) -- ++(-0.1,0.1);
\draw[->]
	(-0.7,0) -- ++(-0.1,-0.1);

\begin{scope}[font=\scriptsize]

\node at (0.7,0.9) {$a$};
\node at (1.6,0) {$b$};	
\node at (-0.7,-0.5) {$b$};
\node at (0.85,0.1) {$c$};
\node at (-0.85,0.1) {$x$};
\node at (0.7,-0.9) {$y$};

\node at (0,0.9) {$V_t$};
\node at (1.55,0.85) {$V_s$};
\node at (1.55,-0.85) {$X$};

\end{scope}

\draw[->, very thick]
	(2.2,0) -- ++(1,0);

\begin{scope}[xshift=5cm]

\draw
	(0,-0.7) -- (-1.4,0.7) -- (0,0.7) -- (0,-0.7)
	(1.4,0.7) -- (1.4,-0.7);

\draw[dashed]
	(0,0.7) -- (1.4,0.7)
	(0,-0.7) -- (1.4,-0.7);
	
\draw[->]
	(-0.7,0.7) -- ++(-0.1,0);
\draw[->]
	(0,0) -- ++(0,0.1);
\draw[->]
	(-0.7,0) -- ++(-0.1,0.1);

\begin{scope}[font=\scriptsize]

\node at (-0.7,0.9) {$a$};
\node at (0.2,0) {$b$};
\node at (-0.55,0.1) {$c$};
\node at (1.6,0) {$c$};

\node at (0,0.9) {$V_s$};
\node at (-1.55,0.85) {\scriptsize $V_t$};
\node at (0,-0.9) {$X$};

\end{scope}

\end{scope}

\end{scope}
}

\draw[->]
	(-0.7,-0.7) -- ++(-0.1,0);

\node at (0.1,-0.9) {\scriptsize $X$};
\node at (-1.55,-0.85) {\scriptsize $V_s$};

\begin{scope}[xshift=5cm]

\draw[->]
	(1.4,0) -- ++(0,0.1);
\draw[->]
	(0.7,0.7) -- ++(-0.1,0);
\draw[->]
	(0.7,-0.7) -- ++(0.1,0);

\node at (0.7,0.9) {\scriptsize $x$};	
\node at (0.7,-0.9) {\scriptsize $y$};
\node at (1.55,0.85) {\scriptsize $V_t$};
\node at (1.55,-0.85) {\scriptsize $X$};

\end{scope}

\begin{scope}[yshift=-2.5cm]

\draw[->]
	(-0.7,-0.7) -- ++(0.1,0);

\node at (0.1,-0.9) {\scriptsize $V_s$};
\node at (-1.55,-0.85) {\scriptsize $X$};

\begin{scope}[xshift=5cm]

\draw[->]
	(1.4,0) -- ++(0,-0.1);
\draw[->]
	(0.7,0.7) -- ++(0.1,0);
\draw[->]
	(0.7,-0.7) -- ++(-0.1,0);
	
\node at (0.7,0.9) {\scriptsize $y$};	
\node at (0.7,-0.9) {\scriptsize $x$};
\node at (1.55,0.85) {\scriptsize $X$};
\node at (1.55,-0.85) {\scriptsize $V_t$};

\end{scope}

\end{scope}

\end{tikzpicture}
\caption{Reduce the degree number of $V_s$.}
\label{cutpaste2}
\end{figure}
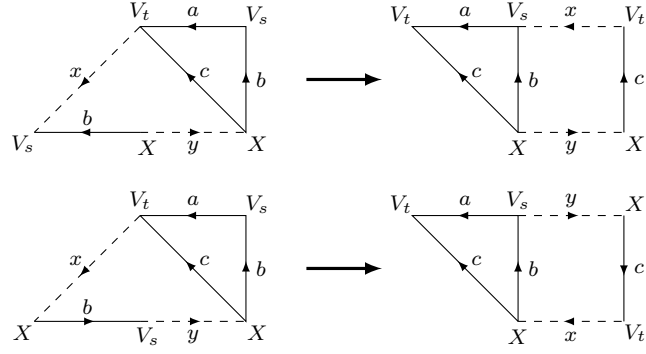

We cut the planar diagram into two pieces along the edge $c$, and glue the two pieces together along the edge $b$. We get the planar diagrams on the right. We find that the degree of $V_s$ is decreased by one, which is balanced by that the degree of $V_t$ is increased by one. The degrees of the other vertices remain the same. By repeating the process, we eventually get a planar diagram with degree $1$ vertices $V_1,V_2,\cdots,V_{v-1}$ and one more vertex $V_v$. We note that $V_1,V_2,\cdots,V_{v-1}$ are the middle points of segments $cc^{-1}$.

The second step is to bring all separated and twisted edge pairs together. Since such a pair $a$ is not $aa^{-1}$,  the end points of $a$ are $V_v$. We get the left of Figure \ref{cutpaste3}, where the segments $x$ and $y$ are not empty. The cut and paste along $c$ gives the right of Figure \ref{cutpaste3}, which is $ccxy^{-1}$. The separated pair $a$ becomes the non-separated pair $c$. 

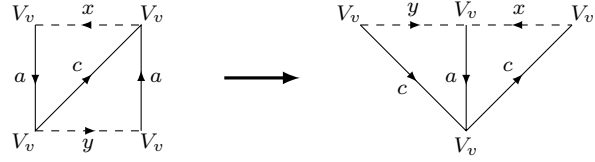
\begin{figure}[htp]
\centering
\begin{tikzpicture}[>=latex,scale=1]

\draw[dashed]
	(-0.7,-0.7) -- (0.7,-0.7)
	(-0.7,0.7) -- (0.7,0.7);
	
\draw
	(-0.7,0.7) -- (-0.7,-0.7) -- (0.7,0.7) -- (0.7,-0.7);

\foreach \a in {0,...,3}
\draw[->, rotate=90*\a]
	(0.7,0) -- ++(0,0.1);

\draw[->]
	(0,0) -- ++(0.05,0.05);
	
\begin{scope}[font=\scriptsize]

\node at (0.9,0) {$a$};
\node at (-0.9,0) {$a$};
\node at (-0.15,0.15) {$c$};
\node at (0,0.9) {$x$};
\node at (0,-0.9) {$y$};

\foreach \a in {1,-1}
\foreach \b in {1,-1}
\node at (0.85*\a,0.85*\b) {$V_v$};

\end{scope}

\draw[->, very thick]
	(1.8,0) -- ++(1,0);

\begin{scope}[xshift=5cm]

\draw
	(-1.4,0.7) -- (0,-0.7) -- (1.4,0.7) 
	(0,-0.7) -- (0,0.7);

\draw[dashed]
	(-1.4,0.7) -- (1.4,0.7);

\draw[->]
	(0.7,0) -- ++(0.05,0.05);
\draw[->]
	(-0.7,0) -- ++(0.05,-0.05);
\draw[->]
	(0,0) -- ++(0,-0.1);
\draw[->]
	(0.7,0.7) -- ++(-0.1,0);
\draw[->]
	(-0.7,0.7) -- ++(0.1,0);
			
\begin{scope}[font=\scriptsize]

\node at (-0.2,0) {$a$};
\node at (0.55,0.15) {$c$};
\node at (-0.85,-0.15) {$c$};
\node at (0.7,0.9) {$x$};
\node at (-0.7,0.9) {$y$};

\node at (-1.55,0.85) {$V_v$};
\node at (1.55,0.85) {$V_v$};
\node at (0,0.9) {$V_v$};
\node at (0,-0.9) {$V_v$};

\end{scope}	

\end{scope}

\end{tikzpicture}
\caption{Bring separated twisted edge pair $a$ together.}
\label{cutpaste3}
\end{figure}

The third step is to bring separated and opposing edge pairs together. Again the end points of such a pair $a$ are $V_v$, and we get the left of Figure \ref{cutpaste4}, where the segments $x$ and $y$ are not empty. Since $V_v$ are the end points of $x$ and $y$, there must be an edge pair $b$ with one $b$ in $x$ and another $b$ in $y$. Since all twisted pairs have been brought together, the pair $b$ is opposing. Then we get the second picture. Cut and paste along $c$ gives the third planar diagram, in which $x'$ is a combination of $x_1$ and $x_2$, and $y'$ is a combination of $y_1$ and $y_2$. Then cut and paste again along $d$ gives the fourth planar diagram $cdc^{-1}d^{-1}y'x'$. 

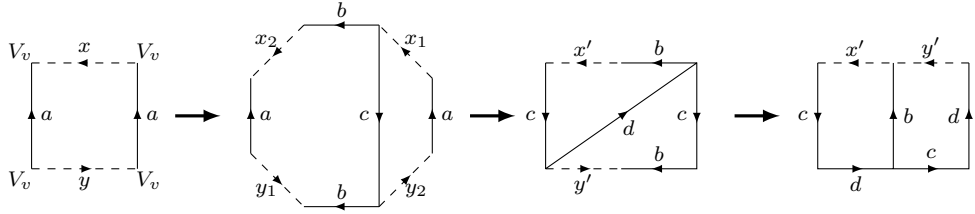
\begin{figure}[htp]
\centering
\begin{tikzpicture}[>=latex,scale=1]


\draw[dashed]
	(-0.7,-0.7) -- (0.7,-0.7)
	(-0.7,0.7) -- (0.7,0.7);
	
\draw
	(-0.7,0.7) -- (-0.7,-0.7) 
	(0.7,0.7) -- (0.7,-0.7);

\foreach \a in {-1,0,1}
\draw[->, rotate=90*\a]
	(0.7,0) -- ++(0,0.1);

\draw[->]
	(-0.7,0) -- ++(0,0.1);
		
\begin{scope}[font=\scriptsize]

\node at (0.9,0) {$a$};
\node at (-0.5,0) {$a$};
\node at (0,0.9) {$x$};
\node at (0,-0.9) {$y$};

\foreach \a in {1,-1}
\foreach \b in {1,-1}
\node at (0.85*\a,0.85*\b) {$V_v$};

\end{scope}

\draw[->, very thick]
	(1.2,0) -- ++(0.6,0);
	

\begin{scope}[xshift=3.4cm]

\foreach \a in {0,...,3}
\draw[rotate=90*\a]
	(-22.5:1.3) -- (22.5:1.3);

\foreach \a in {0,...,3}
\draw[dashed, rotate=45+90*\a]
	(-22.5:1.3) -- (22.5:1.3);

\draw
	(-67.5:1.3) -- (67.5:1.3);
	
\foreach \a in {1,-1}
{
\draw[->]
	(1.2*\a,0) -- ++(0,0.1);
\draw[->]
	(0,1.2*\a) -- ++(-0.1,0);
}

\foreach \a in {0,...,3}
\draw[->, rotate=45+90*\a]
	(1.2,0) -- ++(0,0.1);

\draw[->]
	(0.5,0) -- ++(0,-0.1);
	
\begin{scope}[font=\scriptsize]

\node at (1.4,0) {$a$};
\node at (-1,0) {$a$};
\node at (0,1.4) {$b$};
\node at (0,-1) {$b$};
\node at (0.3,0) {$c$};
\node at (45:1.4) {$x_1$};
\node at (135:1.4) {$x_2$};
\node at (225:1.4) {$y_1$};
\node at (-45:1.4) {$y_2$};

\end{scope}	

\end{scope}

\draw[->, very thick]
	(5.1,0) -- ++(0.6,0);


\begin{scope}[xshift=7.1cm]

\draw
	(0,0.7) -- (1,0.7) -- (1,-0.7) -- (0,-0.7)
	(-1,0.7) -- (-1,-0.7)
	(1,0.7) -- (-1,-0.7);

\draw[dashed]
	(-1,0.7) -- (0,0.7) 
	(-1,-0.7) -- (0,-0.7);

\foreach \a in {1,-1}
{
\draw[->]
	(1*\a,0) -- ++(0,-0.1);
\draw[->]
	(0.5,0.7*\a) -- ++(-0.1,0);	
\draw[->]
	(-0.5,0.7*\a) -- ++(-0.1*\a,0);	
}

\draw[->]
	(0,0) -- ++(0.1,0.07);	
	
\begin{scope}[font=\scriptsize]

\node at (0.5,0.9) {$b$};
\node at (0.5,-0.5) {$b$};
\node at (-1.2,0) {$c$};
\node at (0.8,0) {$c$};
\node at (-0.5,0.9) {$x'$};
\node at (-0.5,-0.9) {$y'$};
\node at (0.1,-0.15) {$d$};
	
\end{scope}

\end{scope}

\draw[->, very thick]
	(8.6,0) -- ++(0.6,0);


\begin{scope}[xshift=10.7cm]

\draw
	(-1,0.7) -- (-1,-0.7) -- (1,-0.7) -- (1,0.7)
	(0,0.7) -- (0,-0.7);

\draw[dashed]
	(-1,0.7) -- (1,0.7);

\foreach \a in {1,-1}
{
\draw[->]
	(1*\a,0) -- ++(0,0.1*\a);
\draw[->]
	(0.5*\a,-0.7) -- ++(0.1,0);	
\draw[->]
	(0.5*\a,0.7) -- ++(-0.1,0);	
}

\draw[->]
	(0,0) -- ++(0,0.1);	

\begin{scope}[font=\scriptsize]

\node at (0.2,0) {$b$};
\node at (-1.2,0) {$c$};
\node at (0.8,0) {$d$};
\node at (-0.5,-0.9) {$d$};
\node at (0.5,-0.5) {$c$};
\node at (-0.5,0.9) {$x'$};
\node at (0.5,0.9) {$y'$};
	
\end{scope}

\end{scope}

\end{tikzpicture}
\caption{Separated opposing edge pair $a$ becomes $cdc^{-1}d^{-1}$.}
\label{cutpaste4}
\end{figure}

Now the planar diagram is a combination of $aa$, $aba^{-1}b^{-1}$ and $cc^{-1}$. The fourth step is to bring all $cc^{-1}$ together. This can be accomplished  by showing $cc^{-1}$ and any $u=aa$ or $aba^{-1}b^{-1}$ can be exchanged. In other words, $ucc^{-1}x$ can be changed to $dd^{-1}ux$. See Figure \ref{cutpaste5}.

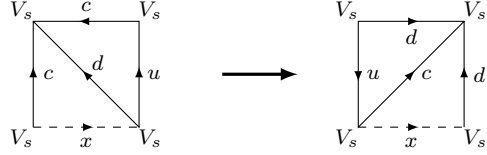
\begin{figure}[htp]
\centering
\begin{tikzpicture}[>=latex,scale=1]

\draw
	(-0.7,-0.7) -- (-0.7,0.7) -- (0.7,0.7) -- (0.7,-0.7) -- (-0.7,0.7);
	
\draw[dashed]
	(-0.7,-0.7) -- (0.7,-0.7);

\foreach \a in {-1,0,1}
\draw[->, rotate=90*\a]
	(0.7,0) -- ++(0,0.1);

\draw[->]
	(0,0) -- ++(-0.05,0.05);
\draw[->]
	(-0.7,0) -- ++(0,0.1);
		
\begin{scope}[font=\scriptsize]

\node at (0.9,0) {$u$};
\node at (-0.5,0) {$c$};
\node at (0.15,0.15) {$d$};
\node at (0,0.9) {$c$};
\node at (0,-0.9) {$x$};

\foreach \a in {1,-1}
\foreach \b in {1,-1}
\node at (0.85*\a,0.85*\b) {$V_s$};

\end{scope}

\draw[->, very thick]
	(1.8,0) -- ++(1,0);

\begin{scope}[xshift=4.3cm]

\draw
	(0.7,0.7) -- (-0.7,-0.7) -- (-0.7,0.7) -- (0.7,0.7) -- (0.7,-0.7);
	
\draw[dashed]
	(-0.7,-0.7) -- (0.7,-0.7);

\draw[->]
	(0,0) -- ++(0.05,0.05);
\draw[->]
	(0.7,0) -- ++(0,0.1);
\draw[->]
	(0,-0.7) -- ++(0.1,0);
\draw[->]
	(0,0.7) -- ++(0.1,0);
\draw[->]
	(-0.7,0) -- ++(0,-0.1);
			
\begin{scope}[font=\scriptsize]

\node at (0.2,0) {$c$};
\node at (0.9,0) {$d$};
\node at (0,0.5) {$d$};
\node at (0,-0.9) {$x$};
\node at (-0.5,0) {$u$};

\foreach \a in {1,-1}
\foreach \b in {1,-1}
\node at (0.85*\a,0.85*\b) {$V_s$};

\end{scope}	

\end{scope}

\end{tikzpicture}
\caption{Exchange $cc^{-1}$ and $aa$ or $aba^{-1}b^{-1}$.}
\label{cutpaste5}
\end{figure}

Now the planar diagram is a combination of the words
\[
aa,\qquad aba^{-1}b^{-1},\qquad
c_1c_1^{-1}\cdots c_{v-1}c_{v-1}^{-1}.
\]
It remains to show that $aba^{-1}b^{-1}ccx$ can be changed to
$a_1a_1a_2a_2a_3a_3x$. A cut-and-paste procedure transforms
$aba^{-1}b^{-1}cc$ into $a_1a_1a_2a_2a_3a_3$, giving the claim.

\end{proof}

Combining Lemmas \ref{symmetry} and \ref{standard}, we get a precise description of the symmetry group of the universal tiling.

\begin{figure}[htp]
\centering
\begin{tikzpicture}[>=latex,scale=1]


\foreach \a in {-1,...,5}
\draw
	(0,0) -- (45*\a:1.2);

\foreach \a in {0,3,4,5}
\draw[->]
	(45*\a:0.8) -- ++(45*\a:0.1);

\foreach \a in {-1,1,2}
\draw[->]
	(45*\a:0.8) -- ++(45*\a:-0.1);
			
\fill
	(0,0) circle (0.05);

\begin{scope}[font=\scriptsize]

\foreach \a/\b in {-1/1,0/4,1/3,2/2,3/5,4/7}
\node at (22.5+45*\a:0.4) {\b};

\node at (-57:0.8) {$x$};
\node at (-12:0.8) {$a$};
\node at (33:0.8) {$b$};
\node at (102:0.8) {$a$};
\node at (147:0.8) {$b$};
\node at (168:0.8) {$c$};
\node at (213:0.8) {$y$};

\end{scope}

\begin{scope}[shift={(-10cm,0.8cm)}]

\draw
	(0,0) -- (8,0);

\draw[dashed]
	(0,0) to[out=-60, in=180] (4,-1.5) to[out=0, in=240] (8,0);
	
\foreach \a in {0,...,8}
\fill
	(\a,0) circle (0.05);

\foreach \a in {0,2,5,6,7}
\draw[->]
	(0.5+\a,0) -- ++(-0.1,0);

\foreach \a in {1,3,4}
\draw[->]
	(0.5+\a,0) -- ++(0.1,0);

\node at (4, -0.8) {$P$};

\begin{scope}[font=\scriptsize]

\foreach \a in {1,...,7}
\node at (8-\a,-0.2) {\a};

\node at (0.5,-0.2) {$y$};
\node at (1.5,0.2) {$c$};
\node at (2.5,-0.2) {$c$};
\node at (3.5,-0.2) {$b$};
\node at (4.5,0.2) {$a$};
\node at (5.5,0.2) {$b$};
\node at (6.5,-0.2) {$a$};
\node at (7.5,0.2) {$x$};

\end{scope}

\end{scope}


\begin{scope}[yshift=-2.6cm]

\foreach \a in {0,...,4}
\draw
	(0,0) -- (45*\a:1.2);

\foreach \a in {1,3,4}
\draw[->]
	(45*\a:0.8) -- ++(45*\a:0.1);

\foreach \a in {0,2}
\draw[->]
	(45*\a:0.8) -- ++(45*\a:-0.1);
			
\fill
	(0,0) circle (0.05);

\begin{scope}[font=\scriptsize]

\foreach \a/\b in {0/1,1/2,2/3,3/5}
\node at (22.5+45*\a:0.4) {\b};

\node at (-12:0.8) {$x$};
\node at (33:0.8) {$a$};
\node at (78:0.8) {$a$};
\node at (123:0.8) {$c$};
\node at (168:0.8) {$y$};

\end{scope}

\end{scope}

\begin{scope}[shift={(-9cm,-1.4cm)}]

\draw
	(0,0) -- (6,0);

\draw[dashed]
	(0,0) to[out=-60, in=180] (3,-1.5) to[out=0, in=240] (6,0);
	
\foreach \a in {0,...,6}
\fill
	(\a,0) circle (0.05);

\foreach \a in {0,2,3,4,5}
\draw[->]
	(0.5+\a,0) -- ++(-0.1,0);

\draw[->]
	(1.5,0) -- ++(0.1,0);

\node at (3, -0.8) {$P$};

\begin{scope}[font=\scriptsize]

\foreach \a in {1,...,5}
\node at (6-\a,-0.2) {\a};

\node at (0.5,-0.2) {$y$};
\node at (1.5,0.2) {$c$};
\node at (2.5,-0.2) {$c$};
\node at (3.5,0.2) {$a$};
\node at (4.5,-0.2) {$a$};
\node at (5.5,0.2) {$x$};

\end{scope}

\end{scope}
		
\end{tikzpicture}
\caption{Standard symmetry group of the universal tiling.}
\label{cutpaste6}
\end{figure}
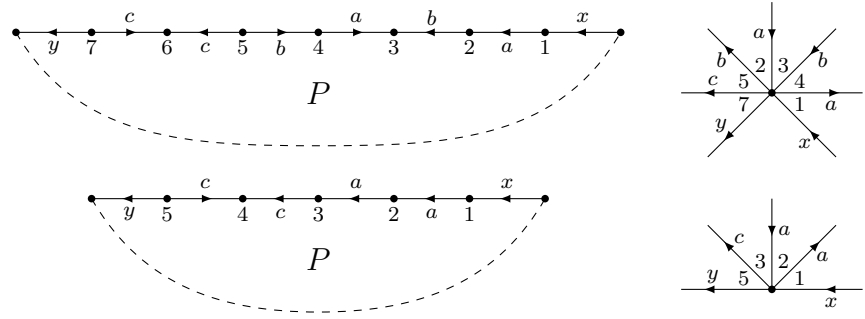

\begin{theorem}\label{standard_symmetry}
Let ${\mc T}$ be a tiling of $S$ by congruent copies of $P$ associated to planar diagram $D$, with vertices $V_1^{k_1},V_2^{k_2},\dots, V_{p+1}^{k_{p+1}}$. Then the universal cover tiling $\tilde{\mc T}$ of $\tilde{S}$ has the symmetry group for orientable $S_D$
\begin{align*}
\Gamma
&=\langle a_1,b_1,\dots,a_g,b_g,c_1,\dots,c_{p+1}\; | \\
&\qquad c_1^{k_1},\dots,c_{p+1}^{k_{p+1}},c_1\cdots c_{p+1}a_1b_1a_1^{-1}b_1^{-1}\cdots a_gb_ga_g^{-1}b_g^{-1}\rangle,
\end{align*}
and the symmetry group for nonorientable $S_D$
\[
\Gamma
=\langle a_1,\dots,a_g,c_1,\dots,c_{p+1}\; | \;
c_1^{k_1},\dots,c_{p+1}^{k_{p+1}},c_1\cdots c_{p+1}a_1^2\cdots a_g^2\rangle.
\]
\end{theorem}

\section{Existence of Even-Sided  Tilings}

By Lemma \ref{standard}, any tiling by  congruent most curvilinear polygons of even number of sides is equivalent to a tiling by congruent most curvilinear polygons associated to one of the following standard planar diagrams:
\begin{itemize}
\item $a_1b_1a_1^{-1}b_1^{-1}\cdots a_gb_ga_g^{-1}b_g^{-1}c_1c_1^{-1}\cdots c_pc_p^{-1}$ ($g\ge1,p\ge0$).
\item $a_1a_1\cdots a_ga_gc_1c_1^{-1}\cdots c_pc_p^{-1}$ ($g\ge1,p\ge0$).
\item $c_1c_1^{-1}\cdots c_pc_p^{-1}$ ($p>0$).
\end{itemize}

Note that in each case, there are $p+1$ vertices $V_1,V_2,...,V_{p+1}$ in the planar diagram. We call the above classification planar diagram of first kind, second kind and third kind. In this section, we deal with the following question:

\medskip

\noindent\textbf{Question.} Let $D$ be a planar diagram with vertices
$V_1,\ldots,V_{p+1}$, and let $k_1,\ldots,k_{p+1}\geq2$ be integers.
For which surfaces $S$ does there exist a tiling by congruent most
curvilinear polygons associated to $D$ with vertex types
$V_1^{k_1},\ldots,V_{p+1}^{k_{p+1}}$?
\medskip

\noindent{\textbf{Remark.}} We always assume $k_i\ge2$ in the above question because whenever there is some $k_i=1$, we may consider the planar diagram $D'$ which reduces $p$ by one in $D$.

If we can find such a  tiling on $S$ with $d$ tiles, by Lemma 2, this is equivalent to the existence of a branched covering $S \to  S_D$ of degree $d$, with branched data $\mathscr{D}=\{A_1,A_2,...,A_{p+1}\}$, where $A_i=[k_i,k_i,...,k_i]$ are partitions of $d$. So given $S$, $D$ and $k_1,k_2,\dots,k_p,k_{p+1}$, the above question is equivalent to whether one can find suitable $d$, such that $d$ has above partitions $\mathscr{D}$ and  $S$ is a branched cover over surface $S_D$ with the branched data $\mathscr{D}$. In fact, this is the special case of Hurwitz existence problem, whose branched data at a point as a partition of $d$ can be arbitrary. 

In general, there are some easy and well-known necessary conditions for a branched covering, called Hurwitz conditions: If $\phi:\tilde{S} \to S $ is a degree $d$ branched covering of closed, connected surfaces with branch data $\mathscr{D}=\{A_1,...,A_m\}$,  then the total branching $\nu(\phi)=\nu(\mathscr{D})=\sum_{i=1}^m\nu(A_i)$ is even; where if $A=[a_1,...,a_r]$ is a partition of $d$, then $\nu(A)=d-r=\sum_{i=1}^r(a_i-1)$. If $S$ is orientable, then so is $\tilde{S}$; if $d$ is odd and $S$ is nonorientable, then $\tilde{S}$ is also nonorientable. And finally, by Riemann-Hurwitz formula, $\chi(\tilde{S})=d\chi(S)-\nu(\phi)$.

If $\mathscr{D}=\{A_1,...,A_m\}$ is a collection of partitions of $d$, we say simply $\mathscr{D}$ is \textit{realizable }if $\mathscr{D}$ arises as the branch data of a connected branched covering of $\bb S^2$. For an element $\alpha$ in symmetric group $\mathfrak{S}_d$ of $d$ element, $A$ is a partition of $d$, we say $\alpha\in A$ if cycles of $\alpha$ give the same partition with $A$. We also use $||\alpha||$ to represent the partition given by $\alpha$, that is $||\alpha||=A$ if $\alpha \in A$. The following  lemma is the algebraic conditions for the existence of a connected $d$-fold covering (See Lemma 2.1 in \cite{article1}). 

\begin{lemma}\label{lemma2.1}
Let \( S \) be a closed, connected surface and \( \mathscr{D} = \{A_1, \ldots, A_n\} \) be a collection of partitions of \( d \) (repetitions allowed). Then a necessary and sufficient condition that there exist a \( d \)-fold branched covering \( \phi: \tilde{\Sigma} \to \Sigma \) with \( \tilde{\Sigma} \) closed and connected and with branch data \( \mathscr{D} \) is as follows:

\begin{enumerate}
\item[(a)] For \( S = \bb S^2 \), there exist elements \( \alpha_i \in \mathfrak{S}_d \), \( \alpha_i \in A_i \), such that \( \alpha_1 \cdots \alpha_n = 1 \) and \( \langle \alpha_1, \ldots, \alpha_n \rangle \) acts transitively on \( \{1, \ldots, d\} \).

\item[(b)] For \( S = g\bb T^2 \), a connected sum of \( g \geq 1 \) copies of \( \bb T^2\), there exist elements \( \alpha_i \in \mathfrak{S}_d \), \( \alpha_i \in A_i \), and \( \beta_j, \gamma_j \in \mathfrak{S}_d \), \( 1 \leq j \leq g \), such that
\[
\alpha_1 \cdots \alpha_n \beta_1 \gamma_1 \beta_1^{-1} \gamma_1^{-1} \cdots \beta_g \gamma_g \beta_g^{-1} \gamma_g^{-1} = 1
\]
and \( \langle \alpha_1, \ldots, \alpha_n, \beta_1, \gamma_1, \ldots, \beta_g, \gamma_g \rangle \) acts transitively on \( \{1, \ldots, d\} \).

\item[(c)] For \( S = g \bb P^2\), a connected sum of \( g \geq 1 \) projective planes \(\bb P^2 \), there exist elements \( \alpha_i \in \mathfrak{S}_d \), \( \alpha_i \in A_i \), and \( \beta_j \in \mathfrak{S}_d \), \( 1 \leq j \leq g \), such that \( \alpha_1 \cdots \alpha_n \beta_1^2 \cdots \beta_g^2 = 1 \) and \( \langle \alpha_1, \ldots, \alpha_n, \beta_1, \ldots, \beta_g \rangle \) acts transitively on \( \{1, \ldots, d\} \). 
\end{enumerate}
\end{lemma}

\noindent{\textbf{Remark.}} Let $S-n\cdot\{ pt\}$ be the surface $S$ deletes $n$ branched points. Then by Lemma \ref{lemma2.1}, each branched cover defines a map $\theta:\pi_1(\Sigma-n\cdot\{pt\})\to \mathfrak{S}_d$. For example, if $S=g\bb T^2$, then:\\
$\pi_1(g\bb T^2-n\cdot\{pt\})=<e_1,...,e_n,a_1,...,a_g,b_1,...,b_g| e_1...e_n[a_1,b_1]...[a_g,b_g]=1>$. We will keep this notation and use $e_i$ to represent elements in $\pi_1(S-n\cdot\{ pt\})$ around deleted points.

For natural numbers $k_1,...,k_{p+1}\ge2$, let $\lambda=LCM(k_1,...,k_{p+1})$, $B=\sum_{i=1}^{p+1}(\lambda-\frac{\lambda}{k_i})$. Define $d'=2\lambda$ if $B$ is odd,  $d'=\lambda$ if $B$ is even. Note that $d'$ is the least positive number such that condition 1 and 2 in Theorem \ref{case12} hold.

\begin{lemma}\label{realizability}
Suppose $k_1,...,k_{p+1}\ge2$ are natural numbers, $A_i=[k_i,...,k_i]$ for $1\le i \le p+1$ is a partition of $d'$. Then the collection $\mathscr{D}=\{A_1,...,A_{p+1}\}$ is realizable if  $p\ge2$. 
\end{lemma}
\begin{proof}
See Theorem (1.4) in \cite{article2}.

\end{proof}

\begin{corollary}\label{core lemma}
Let $k_1,\ldots,k_{p+1}\ge2$ be integers, with $p\ge1$,
and let $d\ge1$ be an integer.
There exists a connected branched covering of $S^2$ of
degree $d$, with branch data
$\mathscr{D}=\{A_1,\ldots,A_{p+1}\}$, where
$A_i=[k_i,\ldots,k_i]$, if and only if
\begin{enumerate}
\item $k_i\mid d$ for every $i=1,\ldots,p+1$;
\item the total branching
\[
\nu(\mathscr{D} )
=\sum_{i=1}^{p+1}\left(d-\frac{d}{k_i}\right)
\]
is even;
\item $\nu(\mathscr{D})\ge2d-2$.
\end{enumerate}
\end{corollary}

\begin{proof}
Necessity follows from the Hurwitz conditions.
For sufficiency, If $p=1$, conditions 1 and 3 give
$d/k_1+d/k_2\le2$, hence $k_1=k_2=d$.
Taking a $d$-cycle and its inverse realizes
$A_1=A_2=[d]$ by Lemma \ref{lemma2.1}(a).
Henceforth assume $p\ge2$. Conditions 1 and 2 imply $d=sd'$
for some positive integer $s$.
By Lemma \ref{realizability}, there is a connected branched covering
$f_0:S_0\to S^2$ of degree $d'$ with the prescribed
local degrees $k_i$.
The Riemann--Hurwitz formula gives
\[
s\chi(S_0)=2d-\nu(\mathscr{D})\le2.
\]
If $\chi(S_0)>0$, then $\chi(S_0)=2$, so $s=1$,
and $f_0$ gives the required covering.
Otherwise, take a connected unbranched covering
$h:S_1\to S_0$ of degree $s$.
Then $f_0\circ h$ has degree $d$ and branch data
$\mathscr{D}$, as required.
\end{proof}

\begin{theorem}\label{case12}
Let $S$ be a closed, connected topological surface, and let $D$
be a planar diagram with associated surface $S_D$.
Let $V_1,\ldots,V_{p+1}$ be the vertices of $S_D$, and let
$k_1,\ldots,k_{p+1}\geq 2$ and $d\geq 1$ be integers.

There exists a metric of constant  curvature on $S$
admitting a tiling by $d$ congruent most curvilinear polygons
associated to $D$, with vertex types
\[
V_1^{k_1},\ldots,V_{p+1}^{k_{p+1}},
\]
if and only if the following conditions hold:
\begin{enumerate}
    \item $k_i\mid d$ for every $i=1,\ldots,p+1$.

    \item The total branching
    \[
    \nu(\mathscr{D})=\sum_{i=1}^{p+1}
    \left(d-\frac{d}{k_i}\right)
    \]
    is even.

    \item The Euler characteristics satisfy
    \[
    \chi(S)=d\,\chi(S_D)-\nu(\mathscr{D}).
    \]

    \item If $S_D$ is orientable, then $S$ is orientable.

    \item If $S_D$ is nonorientable and $S$ is orientable, then
    \[
    2k_i\mid d
    \qquad\text{for every }i=1,\ldots,p+1.
    \]
\end{enumerate}
\end{theorem}

\begin{proof}
By Lemma \ref{lem:branched-cover}, it suffices to consider connected branched
coverings $S\to S_D$ of degree $d$ with branch data
\[
\mathscr{D}=\{A_1,\ldots,A_{p+1}\},
\qquad A_i=[k_i,\ldots,k_i].
\]
Conditions 1--4 are necessary by the Hurwitz conditions.
Condition 5 follows from Proposition 2.7 in \cite{article1}, since
each $A_i$ must split into two partitions of $d/2$.

For sufficiency, suppose first that $\chi(S_D)\le0$.
The required covering exists by Propositions 3.3 and 3.4 in \cite{article1}.

If $S_D=S^2$, then $S_D$ has at least two vertices.
Since $S$ is orientable, condition 3 gives
\[
\nu(\mathscr{D})=2d-\chi(S)\ge2d-2.
\]
 Corollary \ref{core lemma} therefore gives the required covering.

Now suppose $S_D=P^2$.
If $S$ is nonorientable, condition 3 gives
\[
\nu(\mathscr{D})=d-\chi(S)\ge d-1.
\]
The required covering exists by Theorem 5.1 in \cite{article1}.

If $S$ is orientable, let $d_1=d/2$.
By condition 5, each $A_i$ splits as
$A_i=[A_i',A_i'']$, where $A_i'$ and $A_i''$
are partitions of $d_1$ consisting of entries $k_i$.
The collection
\[
\mathscr{D}'
=\{A_1',A_1'',\ldots,A_{p+1}',A_{p+1}''\}
\]
satisfies
\[
\nu(\mathscr{D'})=\nu(\mathscr{D})
=d-\chi(S)\ge2d_1-2,
\]
and its total branching is even.
By Corollary \ref{core lemma} , $\mathscr{D'}$ is realizable
over $S^2$. Proposition 2.7 in \cite{article1} then gives the
required covering of $P^2$ with orientable source.

In every case, the source has the same Euler characteristic
and orientability as $S$, and hence is homeomorphic to $S$.
The conclusion follows from Lemma \ref{lem:branched-cover}.
\end{proof}

\section{Existence of Odd-Sided Tilings}

Like Lemma \ref{standard}, for tilings by congruent most curvilinear polygons of odd number of sides, we have the following lemma:

\begin{lemma}\label{standard for odd}
Any tiling by  congruent most curvilinear polygons of odd number of sides is equivalent to a tiling by congruent most curvilinear polygons associated to one of the following standard planar diagrams:
\begin{itemize}
\item $a_1b_1a_1^{-1}b_1^{-1}\cdots a_gb_ga_g^{-1}b_g^{-1}c_1c_1^{-1}\cdots c_pc_p^{-1}r$ ($g\ge1,p\ge0$).
\item $a_1a_1\cdots a_ga_gc_1c_1^{-1}\cdots c_pc_p^{-1}r$ ($g\ge1,p\ge0$).
\item $c_1c_1^{-1}\cdots c_pc_p^{-1}r$ ($p>0$).
\end{itemize}
where $r$ is the unique $r-$edge in the polygon.
\end{lemma}

\begin{proof}
Apply the cut-and-paste argument of Lemma \ref{standard}. Choose the vertex incident to the $r$-side as the distinguished vertex whose degree is not reduced to one. Each cut joins two corners, and each gluing identifies
a pair of matching sides. Thus the unique $r$-side
is neither cut nor used for gluing, and the resulting
polygon still has exactly one $r$-side.
The same reductions give the stated standard forms.
\end{proof}

Given any standard planar diagram $D$  in  Lemma \ref{standard for odd}, in each case, there are $p+1$ vertices $V_1,V_2,...,V_p,V_{p+1}$ in the planar diagram. If we replace the edge $r$ in $D$ by an opposing pair $c_{p+1}c_{p+1}^{-1}$, see Figure \ref{replace1}, we get a standard planar diagram $D'$ which has $p+2$ vertices. Given natural numbers $k_1,k_2,\dots,k_p,k_{p+1}$, under the above convention, we also have the following lemma for tiling by  congruent most curvilinear polygons of odd number of sides:

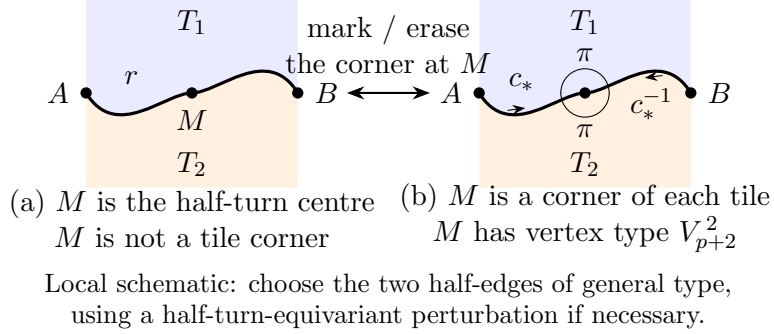
\begin{figure} 
    \centering
\begin{tikzpicture}[
  >=Stealth,
  font=\small,
  line cap=round,
  edge/.style={very thick},
  point/.style={circle,fill,inner sep=1.5pt}
]
  \foreach \offset in {0,5.2} {
    \begin{scope}[xshift=\offset cm]
      \path[fill=blue!8]
        (-1.4,0)
        .. controls (-1,-.65) and (-.35,0) .. (0,0)
        .. controls (.35,0) and (1,.65) .. (1.4,0)
        -- (1.4,1.25) -- (-1.4,1.25) -- cycle;

      \path[fill=orange!12]
        (-1.4,0)
        .. controls (-1,-.65) and (-.35,0) .. (0,0)
        .. controls (.35,0) and (1,.65) .. (1.4,0)
        -- (1.4,-1.25) -- (-1.4,-1.25) -- cycle;

      \draw[edge]
        (-1.4,0)
        .. controls (-1,-.65) and (-.35,0) .. (0,0)
        .. controls (.35,0) and (1,.65) .. (1.4,0);

      \node[point,label=left:$A$] at (-1.4,0) {};
      \node[point,label=right:$B$] at (1.4,0) {};

      \node at (0,.94) {$T_1$};
      \node at (0,-.94) {$T_2$};
    \end{scope}
  }

  \node[point,label=below:$M$] at (0,0) {};
  \node at (-.8,.2) {$r$};

  \node[align=center] at (0,-1.65) {
    (a) $M$ is the half-turn centre\\
    $M$ is not a tile corner
  };

  \draw[<->,thick] (2.05,0) -- (3.15,0);

  \node[align=center] at (2.6,.65) {
    mark / erase\\
    the corner at $M$
  };

  \begin{scope}[xshift=5.2cm]
    \node[point] at (0,0) {};

    \draw
      (.32,0)
      arc[start angle=0,end angle=180,radius=.32];

    \draw
      (-.32,0)
      arc[start angle=180,end angle=360,radius=.32];

    \node at (0,.5) {$\pi$};
    \node at (0,-.5) {$\pi$};

    \node at (-.84,.14) {$c_*$};
    \node at (.88,-.15) {$c_*^{-1}$};

    \draw[->] (-1.02,-.23) -- (-.79,-.2);
    \draw[->] (1.02,.23) -- (.79,.2);

    \node[align=center] at (0,-1.65) {
      (b) $M$ is a corner of each tile\\
      $M$ has vertex type $V_{p+2}^{\,2}$
    };

  \end{scope}

  \node[align=center,font=\footnotesize] at (2.6,-2.75) {
    Local schematic: choose the two half-edges of general type,\\
    using a half-turn-equivariant perturbation if necessary.
  };
\end{tikzpicture}
\caption{Replacing the $r$-edge by an opposing pair $c_{p+1}c_{p+1}^{-1}$.}
\label{replace1}
\end{figure}

\begin{lemma}
Let $S$ be a surface, and let $D$ be a standard odd-sided planar diagram
with vertices $V_1,\dots,V_{p+1}$. Let $D'$ be obtained from $D$
by replacing the $r$-side with the opposing pair
$c_{p+1}c_{p+1}^{-1}$. The following two statements are equivalent:
\begin{enumerate}
\item $S$ admits a tiling by congruent most curvilinear polygons
associated to $D$, with vertex types
$V_1^{k_1},\dots,V_{p+1}^{k_{p+1}}$.
\item $S$ admits a tiling by congruent most curvilinear polygons
associated to $D'$, with vertex types
$V_1^{k_1},\dots,V_{p+1}^{k_{p+1}},V_{p+2}^{2}$.
\end{enumerate}
Thus the existence problem for odd-sided polygons reduces
to the even-sided case.
\end{lemma}

\begin{proof}
$(1)\Rightarrow(2)$.
Mark the half-turn centre of each $r$-edge as a corner,
thereby replacing the $r$-side by the opposing pair
$c_{p+1}c_{p+1}^{-1}$. If necessary, apply a small half-turn-equivariant perturbation,
consistently across all tiles, so that the half-edges are of general
type and the new edge pair is noncongruent to the existing pairs. This gives the required tiling
associated to $D'$ with the additional vertex type
$V_{p+2}^{2}$.

$(2)\Rightarrow(1)$.
The vertex type $V_{p+2}^{2}$ consists of two corners
of angle $\pi$. By the side matching, the local half-turn
exchanges the two half-edges. Removing the marked corner
therefore combines them into a single $r$-symmetric side
and recovers a tiling associated to $D$.
\end{proof}

\section{Existence of Minimal tilings}

Theorem \ref{case12} gives a necessary and sufficient condition for when a tiling by most curvilinear polygons exists. Among all the tilings, many tilings descend to a tiling with fewer tiles, so our main interest is in minimal tilings . We first introduce some terminology and establish  lemmas that are needed.  \\

Definition:
Given \( 1 < l < d \) and \( l \mid d \), let $\Sigma_d^l$ be the collection of all the partitions of $\{1,2,...,d\}$ into $d/l$ subsets (the blocks) such that each block has $l$ elements.
We call a partition $\Pi$ is a \textit{block decomposition of order \(l\)} for \( \tau \in \mathfrak{S}_d \) if $\Pi\in\Sigma_d^l$ and \( \tau \) induces a well-defined permutation \( \hat{\tau} \) of the blocks in $\Pi$. We also write $\tau \in \Pi$ in this case. We say that \( G < \mathfrak{S}_d \) has a block decomposition of order \(l\) if there is a common such decomposition for all the elements of \(G\). We use $\{a_1,a_2,...,a_l\}$ to denote blocks of $\Pi$. For example, $\Pi=\{\{1,3,4\},\{2,6,5\}\}$ represents an element of $\Sigma_6^3$, it has 2 blocks and the first block contains elements $1,3,4$.

\begin{lemma} \label{lem:5.2}
Let \( \tau \) have cyclic structure \((d_1, \dots, d_m)\). Then \( \tau \) admits a block decomposition of order \(l\) if and only if \((d_1, \dots, d_m)\) can be partitioned into sets \( D_1, \dots, D_t \) and there exist integers \( p_1, \dots, p_t \geq 1 \) such that for all \( j \)
\[
p_j \mid x \quad \forall x \in D_j \quad \text{and} \quad \sum_{x \in D_j} \frac{x}{p_j} = l.
\]
In this case \( \hat{\tau} \) has cyclic structure \((p_1, \dots, p_t)\). (Also see Lemma 5.2 in \cite{article4})
\end{lemma}

\begin{proof}
Suppose the block decomposition exists, and take a cycle \((B_1, \dots, B_p)\) of \( \hat{\tau} \), with \( B_j = \{r_{j,1}, \dots, r_{j,l}\} \). Since \( \tau(B_j) = B_{j+1} \) for \( j = 1, \dots, p-1 \), up to changing notation we have
\[
\tau(r_{1,1}) = r_{2,1} \quad \tau(r_{2,1}) = r_{3,1} \quad \cdots \quad \tau(r_{p-1,1}) = r_{p,1}.
\]
Now we either have \( \tau(r_{p,1}) = r_{1,1} \) or, up to changing notation, \( \tau(r_{p,1}) = r_{1,2} \). Proceeding similarly we see that \( \tau \) contains a cycle
\[
(r_{1,1}, \dots, r_{p,1}, r_{1,2}, \dots, r_{p,2}, \dots, r_{1,b_1}, \dots, r_{p,b_1})
\]
of order \( x_1 = p \cdot b_1 \). Repeating the same argument from \( r_{1,b_1+1} \) and so on we find cycles of length \( x_1, \dots, x_s \) of \( \tau \) such that \( p \mid x_i \) and \( \sum_{i=1}^s \frac{x_i}{p} = l \). The conclusion follows by considering the totality of the cycles of \( \hat{\tau} \), and the opposite implication is proved along the same lines.
\end{proof}

In particular, if $\tau$ has cycle type $[k,k,...,k]$, $\Pi$ is a block decomposition of order $l$ for $\tau$ with $p_1=p_2=...=p_t=k$ in the above lemma,  we say $\Pi$ is a \textit{block decomposition of order $l$ associated to $\tau$}, and write $\tau \prec \Pi$. Clearly if $\tau \prec \Pi$, then $\tau \in \Pi$. 

\begin{corollary} \label{cor5}
Let \( \theta : \pi_1(S-n\cdot\{pt\}) \to \mathfrak{S}_d \) be associated to Lemma \ref{lemma2.1} to a branched covering realizing a branch datum \((\tilde S, S, n, d, (d_{ij}))\). Then the covering is decomposable if and only if \(\operatorname{Im}(\theta)\) has a block decomposition. Moreover, if \(\operatorname{Im}(\theta)\) has a block decomposition of order \( l \) with the decomposition of \( \theta(e_i) \) corresponding to a partition of \( (d_{ij}) \) into \( D_{i1}, \ldots, D_{it_i} \) and integers \( p_{i1}, \ldots, p_{it_i} \) then the covering factors through coverings realizing branch data  (See Corollary 5.3 in \cite{article4})

\[
\left( \tilde S, S', t_1 + \ldots + t_n, l, \left( \left( \frac{x}{p_{ij}} \right)_{x \in D_{ij}} \right)_{i=1,\ldots,n, \; j=1,\ldots,t_i} \right)
\]

\[
\left( S', S, n, d/l, \left( (p_{ij})_{j=1,\ldots,t_i} \right)_{i=1,\ldots,n} \right).
\]
\end{corollary}

We will frequently use the following lemma \cite{article1}:
\begin{lemma}\label{lemma3.1}
If \(\alpha \in \mathfrak{S}_d\) with \(\nu(\alpha) = d - r\), then for any \(t \geq 0\) with \(r + 2t \leq d\) one may express \(\alpha\) as the product of  an \((r + 2t)\)-cycle and a \(d\)-cycle.
\end{lemma}

\begin{proof}
One may take
\[
\alpha = (1, \ldots, a_1)(a_1 + 1, \ldots, a_2) \cdots (a_{r-1} + 1, \ldots, a_r)
\]
where \(1 \leq a_1 < a_2 < \cdots < a_r = d\). Let \(\beta = (1, a_1 + 1, a_2 + 1, \ldots, a_{r-1} + 1, b_1, \ldots, b_{2t})\) where \(b_1 < b_2 < \cdots < b_{2t}\) are not in the set \(\{1, a_1 + 1, \ldots, a_{r-1} + 1\}\). Then \(\gamma = \beta\alpha\) is a \(d\)-cycle: Let \(\beta_0 = (1, a_1 + 1, \ldots, a_{r-1} + 1)\) and \(\beta_1 = (1, b_1, \ldots, b_{2t})\), so \(\beta = \beta_1 \beta_0\); then
\[
 \beta\alpha =\beta_1 \beta_0\alpha= (1, b_1, \ldots, b_{2t})(1, 2, \ldots, d)
\]

$= \begin{aligned}[t]
&\bigl(1,\ldots,b_1-1,b_2,\ldots,b_3-1,
b_4,\ldots,b_{2t-1}-1,b_{2t},\ldots,d,\\
&\qquad b_1,\ldots,b_2-1,b_3,\ldots,b_{2t-2}-1,
b_{2t-1},\ldots,b_{2t}-1\bigr).
\end{aligned}
$

Thus \(\alpha = \beta^{-1}\gamma \) has the required form.
\end{proof}

\begin{corollary}\label{cor3.2}
Every element in the alternating group on \( d \) symbols can be expressed as  

\begin{enumerate}
\item[(i)] a product of two \( d \)-cycles,  
\item[(ii)] a commutator \( \alpha \beta \alpha^{-1} \beta^{-1} \) with \( \alpha \) a \( d \)-cycle,  
\item[(iii)] a product of two squares \( \alpha^2 \beta^2 \) with \( \alpha \beta \) a \( d \)-cycle.
\end{enumerate}

\begin{proof}
If \( \gamma \) lies in the alternating group \( \nu(\gamma) = d - r \) is even. By Lemma \ref{lemma3.1},  
\( \gamma = \sigma \tau \) where \( \sigma \) and \( \tau \) are \( d \)-cycles. Now \( \tau \) is conjugate to \( \sigma^{-1} \), so \( \tau = \beta \sigma^{-1} \beta^{-1} \), so  
$\gamma = \sigma \beta \sigma^{-1} \beta^{-1} $. Finally $\tau $ is conjugate to $\sigma$ so $ \gamma = \sigma ( \alpha \sigma \alpha^{-1}) = (\sigma \alpha)^2 (\alpha^{-1}) ^2 $ for some $\alpha$
and \( (\sigma \alpha) \alpha^{-1} = \sigma \) is a \( d \)-cycle.
\end{proof}

\end{corollary}

Now let $k_1,...,k_{p+1}\ge2$ be integers. Given planar diagram $D$ such that $\chi (S_D)\le 0$, by Theorem \ref{case12}, for any positive integer $s$, there is some tiling on some surface by congruent most curvilinear polygons associated to $D$ with $sd'$ tiles, such that the vertices are $V_1^{k_1},V_2^{k_2},\dots,V_{p+1}^{k_{p+1}}$. Besides, the tiling with $d'$ tiles is minimal. But it is still unknown whether there is any minimal tiling with $sd'$ tiles for $s>1$. Given a tiling with $sd'$ tiles on some surface $S$, by the remark following Lemma \ref{lemma2.1}, there is a map $\theta:\pi_1(S_D-\{V_1,...,V_{p+1}\})\to \mathfrak{S}_{sd'}$. If the tiling is not minimal, suppose the tiling descends through an unbranched covering of degree \(l>1\) to a tiling with \(n\) tiles. The latter has the same vertex types, so \(n=q\lambda\) and its total branching \(q\sum_i(\lambda-\lambda/k_i)\) is even. By the definition of \(d'\), this implies \(d'\mid n\). Since \(sd'=ln\), we obtain \(l\mid s\). So by Corollary \ref{cor5}, there is some $l>1$ and $l\mid s$, such that $Im(\theta)$ has a block decomposition of order $l$, say $\Pi \in \Sigma_{sd'}^l$. Besides, if $\theta(e_i)$ corresponds to a partition of  $(k_i,...,k_i)$ into $D_{i1}, \ldots, D_{it_i}$ and integers $p_{i1}, \ldots, p_{it_i}$, then $p_{i1}= \ldots =p_{it_i}=k_i$, that is $\theta(e_i)\prec \Pi$. Note that in this case, each $D_{ij}$ consists of exactly $l$ disjoint cycles of $\theta(e_i)$. Conversely, suppose such a block decomposition $\Pi$ of order $l>1$
exists. By Corollary \ref{cor5}, the branched covering factors as
$S\to S'\to S_D$, with the first map of degree $l$.
Since $\theta(e_i)\prec\Pi$ for every $i$, all its local degrees
are $k_i/k_i=1$, so $S\to S'$ is unbranched.
The tilewise isometries are compatible with the side identifications and hence induce a metric of the same constant curvature on \(S'\). Since the local degrees of \(S\to S'\) are all one, this metric is smooth at the vertices, and \(S\to S'\) is a local isometry. Thus the original tiling descends to a tiling on \(S'\) with \(sd'/l\) congruent tiles and is not minimal. We have proved the following criterion:

\begin{theorem}\label{minimal criterion}[Minimal tiling criterion]
Let $D$ be a planar diagram with vertices $V_1,\dots,V_{p+1}$
and $\chi(S_D)\leq 0$, and let $k_1,\dots,k_{p+1}\geq 2$
be integers. Let $d'$ be the least positive integer such that
$k_i\mid d'$ for every $i$ and
\[
\sum_{i=1}^{p+1}\left(d'-\frac{d'}{k_i}\right)
\]
is even.

Let $\mathcal{T}$ be a tiling associated to $D$ with $sd'$ congruent
most curvilinear tiles and vertex types
$V_1^{k_1},\dots,V_{p+1}^{k_{p+1}}$, where $s\geq 1$
is an integer. Let
\[
\theta:\pi_1(S_D\setminus\{V_1,\dots,V_{p+1}\})
\longrightarrow \mathfrak{S}_{sd'}
\]
be its associated monodromy representation.

Then $\mathcal{T}$ is not minimal if and only if $\operatorname{Im}(\theta)$
admits a block decomposition $\Pi$ of order $l>1$, with $l\mid s$,
such that $\theta(e_i)\prec\Pi$ for every $i$.
\end{theorem}

Firstly we suppose $p\ge1$, so there are at least two vertices in the planar diagram $D$. We have the following lemma:

\begin{lemma}\label{lemma_1}
Suppose $ A_i=[k_i,...,k_i]$ is a partition of $sd'$. Given $\theta(e_1)\in A_1$. If one of $k_1,k_2$ is bigger than 2, then for all sufficiently large integers $s$, there exist some $\theta(e_2)\in A_2$ such that for any $\Pi \in \Sigma_{sd'}^s$ satisfying $\theta(e_1) \prec \Pi$, $\theta(e_2) \not \prec \Pi $. 

\end{lemma}

\begin{proof}
Firstly we suppose $s\ge2$ is any integer and $l|s$, $l\ge2$. Given $\theta(e_1)\in A_1$, the number of $\Pi\in \Sigma_{sd'}^l$ satisfying $\theta(e_1) \prec \Pi$ is
$$C_1=\frac{(\frac{sd'}{k_1})!}{(l!)^{\frac{sd'}{k_1l}}(\frac{sd'}{k_1l})!}(k_1^{l-1})^{\frac{sd'}{k_1l}}$$
Given any $\Pi$ as above , we want to compute the number of $\theta(e_2)\in A_2$ such that  $\theta(e_2) \prec \Pi$. The number of $l$-blocks in $\Pi$ is $\frac{sd'}{l}$. We divide $\frac{sd'}{l}$ blocks into $\frac{sd'}{lk_2}$ parts, each part has $k_2$ blocks, there are 
$$C_2=\frac{(\frac{sd'}{l})!}{(k_2!)^{\frac{sd'}{lk_2}}(\frac{sd'}{lk_2})!}$$  
ways of such division for $\Pi$. Each division yields 
$$C_3=((k_2-1)!(l!)^{k_2-1})^{\frac{sd'}{lk_2}}$$
different such $\theta(e_2)$. So there are $C_2C_3$ number of $\theta(e_2)$ such that $\theta(e_2)\prec \Pi$. So given $\theta(e_1)$, there are at most $C=C_1C_2C_3$  of $\theta(e_2) \in A_2$ such that there exist some $\Pi$ and $\theta(e_1),\theta(e_2) \prec \Pi$. On the other hand, the total number of $\theta(e_2) \in A_2$ is
$$C'=\frac{(sd')!}{(k_2!)^{\frac{sd'}{k_2}}(\frac{sd'}{k_2})!}(k_2-1)!^{\frac{sd'}{k_2}}$$
Now we compare the sizes of $C$ and $C'$ as $s$ tends to infinity. 
$$\frac{C}{C'}=\frac{  (\frac{sd'}{k_1})!  (k_1^{l-1})^{\frac{sd'}{k_1l}}      (\frac{sd'}{l})!     ((k_2-1)!(l!)^{k_2-1})^{\frac{sd'}{lk_2}}        (k_2!)^{\frac{sd'}{k_2}}(\frac{sd'}{k_2})!           }
{ (l!)^{\frac{sd'}{k_1l}}(\frac{sd'}{k_1l})!      (k_2!)^{\frac{sd'}{lk_2}}(\frac{sd'}{lk_2})!        (sd')!          (k_2-1)!^{\frac{sd'}{k_2}}                           }$$
Let $\kappa(s)=\sqrt[sd']{\frac{C(s)}{C'(s)}}$,  assume $l=s$, then 
$$\kappa(s)= k_1^{\frac{s-1}{k_1 s}} \cdot (s!)^{\frac{1}{s}(1-\frac{1}{k_1}-\frac{1}{k_2})} \cdot (k_2)^{\frac{s-1}{s k_2}}  
\cdot \frac{ \left( \frac{sd'}{k_1} \right)!^{1/sd'} \left( \frac{sd'}{s} \right)!^{1/sd'} \left( \frac{sd'}{k_2} \right)!^{1/sd'} }
{ \left( \frac{sd'}{k_1 s} \right)!^{1/sd'} \left( \frac{sd'}{s k_2} \right)!^{1/sd'} \left( (sd')! \right)^{1/sd'} }$$
Using Stirling approximation: $n!\sim \sqrt{2\pi n}(\frac{n}{e})^n$, we get
$$\lim_{s \to \infty} \kappa(s)=d'^{\frac{1}{k_1}+\frac{1}{k_2}-1}$$
So if one of $k_1,k_2$ is bigger than 2, then $\lim_{s \to \infty} \kappa(s)<1$, that is $C<C'$ for $l=s$ big enough. This means in the above assumption,there exist some $\theta(e_2)\in A_2$ such that for any $\Pi \in \Sigma_{sd'}^s$ satisfying $\theta(e_1) \prec \Pi$, $\theta(e_2) \not \prec \Pi $.
\end{proof}

We also need the following lemmas:

\begin{lemma}\label{gptlemma}
Let $s\geq 2$, and set
\[
\alpha=(1,2)(3,4)\cdots(2s-1,2s),
\]
and
\[
\beta=(2,3)(4,5)\cdots(2s-2,2s-1)(2s,1).
\]
If $\Pi\in\Sigma_{2s}^{s}$ satisfies
\[
\alpha\prec\Pi
\qquad\text{and}\qquad
\beta\prec\Pi,
\]
then
\[
\Pi=\{\{1,3,\ldots,2s-1\},\{2,4,\ldots,2s\}\}.
\]
In particular, $\Pi$ is the unique block decomposition of order $s$
associated to both $\alpha$ and $\beta$.
\end{lemma}

\begin{proof}
Since $\Pi\in\Sigma_{2s}^{s}$, it consists of exactly two blocks,
each of cardinality $s$. The condition $\alpha\prec\Pi$ means that
the permutation induced by $\alpha$ on the two blocks has cycle type
$(2)$. Hence $\alpha$ exchanges the two blocks. Similarly, since
$\beta\prec\Pi$, the permutation induced by $\beta$ on the two blocks
also has cycle type $(2)$, and therefore $\beta$ exchanges the two
blocks.

It follows that $\alpha\beta$ preserves each block of $\Pi$. On the
other hand, a direct computation gives
\[
\alpha\beta
=
(2s-1,2s-3,2s-5,\ldots,1)
(2,\,4,\,\ldots,2s).
\]
Thus the two orbits of $\alpha\beta$ are
\[
O=\{1,3,\ldots,2s-1\}
\qquad\text{and}\qquad
E=\{2,4,\ldots,2s\},
\]
and both have cardinality $s$.

Since every block of $\Pi$ is invariant under $\alpha\beta$, each
block must be a union of orbits of $\alpha\beta$. As each block has
cardinality $s$, while each of the two orbits $O$ and $E$ has
cardinality $s$, every block must be either $O$ or $E$. Consequently,
\[
\Pi=\{O,E\}
=\{\{1,3,\ldots,2s-1\},\{2,4,\ldots,2s\}\}.
\]
This proves the uniqueness of $\Pi$.
\end{proof}

\begin{lemma}\label{gpt2}
Let $s\ge3$ be a prime number. Define
\[
\begin{aligned}
\alpha&=(1,2)(3,4)\cdots(4s-1,4s),\\
\beta&=(2,3)(4,5)\cdots(4s-2,4s-1)(4s,1),\\
\gamma&=(1,3)(2,4)(5,6)(7,8)\cdots(4s-1,4s).
\end{aligned}
\]
Then there is no $\Pi\in\Sigma^s_{4s}$ such that
\[
\alpha\prec\Pi,\qquad
\beta\prec\Pi,\qquad
\gamma\prec\Pi.
\]
\end{lemma}

\begin{proof}
Suppose, for contradiction, that such a block
decomposition exists. Write
\[
\Pi=\{B_1,B_2,B_3,B_4\},\qquad |B_i|=s.
\]
Since $\alpha\prec\Pi$, its induced permutation on
the four blocks has cycle type $(2,2)$.
After relabeling the blocks, we may assume that
$1\in B_1$ and
\[
\widehat{\alpha}=(B_1,B_2)(B_3,B_4).
\]
Similarly, $\widehat{\beta}$ is one of
\[
(B_1,B_2)(B_3,B_4),\quad
(B_1,B_3)(B_2,B_4),\quad
(B_1,B_4)(B_2,B_3).
\]

We first show that $\widehat{\beta}\ne\widehat{\alpha}$.
Otherwise, $\widehat{\alpha\beta}=1$, so $\alpha\beta$
preserves every block. However,
\[
\alpha\beta
=(1,4s-1,4s-3,\ldots,3)(2,4,6,\ldots,4s).
\]
Each orbit has length $2s$, whereas each invariant
block has only $s$ elements, a contradiction.

After possibly interchanging $B_3$ and $B_4$,
we may therefore assume that
\[
\widehat{\beta}=(B_2,B_3)(B_4,B_1),
\qquad
\widehat{\alpha\beta}=(B_1,B_3)(B_2,B_4).
\]
Let
\[
O=\{1,3,5,\ldots,4s-1\},\qquad
E=\{2,4,6,\ldots,4s\}.
\]
These are the two orbits of $\alpha\beta$, each
of cardinality $2s$.

Since $1\in B_1$ and $\widehat{\alpha\beta}$
interchanges $B_1$ and $B_3$, we have
$O=B_1\cup B_3$. Moreover, $\alpha\beta$ alternates
between these blocks along its cycle. Hence
\[
B_1=\{1,5,9,\ldots,4s-3\},\qquad
B_3=\{3,7,11,\ldots,4s-1\}.
\]
Likewise, $E=B_2\cup B_4$, and, up to interchanging
$B_2$ and $B_4$,
\[
B_2=\{2,6,10,\ldots,4s-2\},\qquad
B_4=\{4,8,12,\ldots,4s\}.
\]
Thus, up to relabeling the blocks,
\[
\Pi=
\{\{1,5,\ldots,4s-3\},
\{2,6,\ldots,4s-2\},
\{3,7,\ldots,4s-1\},
\{4,8,\ldots,4s\}\}.
\]

Finally, $1,5\in B_1$, whereas
\[
\gamma(1)=3\in B_3,\qquad
\gamma(5)=6\in B_2.
\]
Hence $\gamma(B_1)$ contains elements from two
distinct blocks and cannot itself be a block of
$\Pi$, contradicting $\gamma\prec\Pi$.
\end{proof}

Now we can prove the following result:

\begin{theorem}\label{theorem1}
Given natural numbers $k_1,...,k_{p+1}\ge2$ with $p \ge 1$, suppose $D$ is a planar diagram with $p+1$ vertices such that $\chi (S_D)\le 0$. Then  for every sufficiently large prime number $s$, there exist minimal tilings with $sd'$ tiles by congruent most curvilinear polygons associated to the planar diagram $D$ with vertices $V_1^{k_1},...,V_{p+1}^{k_{p+1}}$ on some surface.

\end{theorem}

\begin{proof}

Let $s\geq 3$ be a sufficiently large prime.
Suppose first that at least one of the $k_i$ is greater
than $2$. After relabeling, assume that $k_2>2$.
Fix any $\theta(e_1)\in A_1$.
By Lemma \ref{lemma_1}, there exist
$\theta(e_2)\in A_2$ such that no
$\Pi\in\Sigma^s_{sd'}$ satisfies both
\[
\theta(e_1)\prec\Pi
\quad\text{and}\quad
\theta(e_2)\prec\Pi.
\]

If $S_D=g\bb T^2$, for $3\le i\le p+1$, take any $\theta(e_i)\in A_i$. The definition of $d'$ easily implies $\theta(e_1)\theta(e_2)...\theta(e_{p+1})$ is an element in the alternating group on $sd'$ symbols. By Corollary \ref{cor3.2}, we can find $a_1,b_1 \in \mathfrak{S}_{sd'}$ such that $\theta(e_1)\theta(e_2)\theta(e_3)...\theta(e_{p+1})[a_1,b_1]=1$ with $a_1$ a $sd'$-cycle. Let $a_j=b_j=1$ for $2\le j\le g$. Then we have $\theta(e_1)\theta(e_2)\theta(e_3)...\theta(e_{p+1})[a_1,b_1]...[a_g,b_g]=1$. By Lemmas \ref{lemma2.1} and \ref{lem:branched-cover} , these permutations yield a tiling associated to $D$ with $sd'$ tiles. By Theorem \ref{minimal criterion}, this tiling is minimal by our construction of $\theta(e_2)$.

If $S_D=g\bb P^2$, for $3\le i\le p+1$, take any $\theta(e_i)\in A_i$. By Corollary \ref{cor3.2}, we can find $a_1,a_2 \in \mathfrak{S}_{sd'}$ such that $\theta(e_1)\theta(e_2)\theta(e_3)...\theta(e_{p+1})a_1^2a_2^2=1$ with $a_1a_2$ a $sd'$-cycle. Let $a_j=1$ for $3\le j\le g$. Then we have $\theta(e_1)\theta(e_2)\theta(e_3)...\theta(e_{p+1})a_1^2a_2^2...a_g^2=1$. By Lemmas \ref{lemma2.1} and \ref{lem:branched-cover} , these permutations yield a tiling associated to $D$ with $sd'$ tiles. By Theorem \ref{minimal criterion}, this tiling is minimal by our construction of $\theta(e_2)$.

It remains to consider the case where all  $k_i$ are equal to $2$. If $p=1$, then $d'=2$, take $\theta(e_1)=(1,2)(3,4)...(2s-1,2s)$, $\theta(e_2)=(2,3)(4,5)...(2s-2,2s-1)(2s,1)$. By Lemma \ref{gptlemma}, the only common block decomposition $\Pi$ of order $l=s$ associated to $\theta(e_1),\theta(e_2)$ is $\Pi=\{\{1,3,5,...,2s-1\},\{2,4,6,...,2s\}\}$. Now $\theta(e_1)\theta(e_2)=(2s-1,2s-3,...,1)(2,4,...,2s)$. If $S_D=g\bb T^2$, by Lemma \ref{lemma3.1} and Corollary \ref{cor3.2}, we can find $u=(2s-1,2,2s-3,2s-5,...,1,4,6,...,2s)$ and $v\in \mathfrak{S}_{2s}$ such that $\theta(e_1)\theta(e_2)=[u^{-1},v]$. Taking $a_1=u$, $b_1=u^{-1}v$, $a_j=b_j=1$ for $2\le j\le g$, then $\theta(e_1)\theta(e_2)[a_1, b_1]...[a_g, b_g]=1$, transitivity follows from $a_1=u$ is a $2s$-cycle. If $S_D=g\bb P^2$, we can find $u=(2s-1,2,2s-3,2s-5,...,1,4,6,...,2s)$ and $w \in \mathfrak{S}_{2s}$ such that $\theta(e_1)\theta(e_2)=(u^{-1}w)^2(w^{-1})^2$. Taking $a_1=w$, $a_2=w^{-1}u$, $a_j=1$ for $3\le j\le g$, then $\theta(e_1)\theta(e_2)a_1^2a_2^2...a_g^2=1$, transitivity follows from $a_1a_2=u$ is a $2s$-cycle.  This gives a tiling associated to $D$ with $2s$ tiles. In both cases, the generated permutation group contains $u$.
Since $u(2s-1)=2$ and $u(3)=1$, the permutation $u$
does not preserve $\Pi$. As $s$ is prime, the minimal
tiling criterion implies that the constructed tiling
is minimal.

If  all $k_i$ are equal to 2 and $p\ge2$, then $d'=2$ or $4$. If $d'=2$, take $\theta(e_1)=(1,2)(3,4)...(2s-1,2s)$, $\theta(e_2)=(2,3)(4,5)...(2s-2,2s-1)(2s,1)$. Then the only common block decomposition $\Pi$ of order $l=s$ associated to $\theta(e_1),\theta(e_2)$ is $\Pi=\{\{1,3,5,...,2s-1\},\{2,4,6,...,2s\}\}$. Take $\theta(e_3)=(1,3)(2,4)(5,6)(7,8)...(2s-1,2s)$, then $\Pi$ is not a block decomposition of order $l=s$ associated to $\theta(e_3)$. A similar argument shows this can give a minimal tiling associated to $D$ with $2s$ tiles. In the case $d'=4$, take $\theta(e_1)=(1,2)(3,4)...(4s-1,4s)$, $\theta(e_2)=(2,3)(4,5)...(4s-2,4s-1)(4s,1)$, $\theta(e_3)=(1,3)(2,4)(5,6)(7,8)\cdots(4s-1,4s)$, using Lemma \ref{gpt2} we get the result.

\end{proof}

Now we assume $p=0$, so the planar diagram $D$ has only one vertex. We assume $\chi (S_D)\le 0$. Given $k\ge2$ is a natural number, so $d'=k$ if $k$ is odd and $d'=2k$ if $k$ is even. For any positive integer $s$, we now investigate whether there exist minimal tilings associated to $D$ with $sd'$ tiles and vertices $V^{k}$. Again we first construct a map $\theta:\pi_1(S_D-\{pts\})\to \mathfrak{S}_{sd'}$, then argue it is minimal by Theorem \ref{minimal criterion}. 

Firstly, we assume $k\ge3$ is odd. Without loss of generality, we may assume $\theta(e_1)=(1,...,k)(k+1,...,2k)...((s-1)k+1,...,sk)$. By Lemma \ref{lemma3.1}, we can write $\theta(e_1)$ to be the product of two $sk$ cycles $\theta(e_1)=\beta^{-1}\gamma$ with $\beta=(1,k+1,2k+1,...,(s-1)k+1,2,3,...,k,k+2,...,2k,...,(s-1)k+2,...,sk)$. We will use $(\beta)_i$ to represent the $i$-th entry in the displayed cycle notation for $\beta$, e.g., $(\beta)_2=k+1$. We will argue that for any $s\ge2$ we can't have a common block decomposition $\Pi$ of order $l$ for $\theta(e_1)$ and $\beta$ such that $\theta(e_1)\prec \Pi$, where $l|s$ and $l\ge2$. 

Otherwise, if there is some above common decomposition $\Pi$ of order $l$, since \(\beta\) is a \(sk\)-cycle, elements whose positions in \(\beta\) differ by a multiple of \(sk/l\) lie in the same block, that is
 for $1\le i \le \frac{sk}{l}$, $(\beta)_i$ and $(\beta)_{i+j\cdot\frac{sk}{l}}$ will be in the same block, $0\le j \le l-1$. So $\Pi$ is totally determined by $\beta$. Next we will show that there exist elements $m,n$, they are both in the same block of $\Pi$ and in the same cycle of $\theta(e_1)$, contradicting the fact that $\theta(e_1)\prec \Pi$.

Because $k\ge3$ and $l\ge2$, we may choose
$j_0=\lfloor l/2\rfloor$ such that
\[
\frac{1}{k}\le\frac{j_0}{l}\le\frac{k-1}{k},
\]
that is,
\[
s\le j_0\frac{sk}{l}\le(k-1)s.
\]
Thus there is an integer $t$, $1\le t\le s$, such that
\[
s+(t-1)(k-2)\le j_0\frac{sk}{l}\le s+t(k-2).
\]
Take $m=(\beta)_t$ and $n=(\beta)_{t+j_0sk/l}$.
Their positions in $\beta$ differ by a multiple of $sk/l$,
so $m,n$ lie in the same block of $\Pi$.
Moreover, $m=(t-1)k+1$, and the above inequalities imply
\[
n\in\{(t-1)k+2,\ldots,tk\}.
\]
Hence $m,n$ are distinct elements of the same cycle of
$\theta(e_1)$, contradicting $\theta(e_1)\prec\Pi$.

Now suppose that \(k\ge4\) is even. In this case \(d'=2k\). Choose
\[
\theta(e_1)
=(1,\ldots,k)(k+1,\ldots,2k)\cdots((2s-1)k+1,\ldots,2sk),
\]
and let
\[
\beta=\bigl(
1,k+1,\ldots,(2s-1)k+1,\,
2,\ldots,k,\,
k+2,\ldots,2k,\ldots,\,
(2s-1)k+2,\ldots,2sk
\bigr).
\]
By Lemma \ref{lemma3.1},
\(\theta(e_1)=\beta^{-1}\gamma\) is a product of two \(2sk\)-cycles.

Suppose that there is a common block decomposition \(\Pi\) of order \(l\) for \(\theta(e_1)\) and \(\beta\), where \(l\mid s\), \(l\ge2\), and \(\theta(e_1)\prec\Pi\). Since \(\beta\) is a \(2sk\)-cycle, elements whose positions in \(\beta\) differ by a multiple of \(2sk/l\) lie in the same block.

Choose \(j_0=\lfloor l/2\rfloor\). Then
\[
\frac1k\le\frac{j_0}{l}\le\frac{k-1}{k},
\qquad
2s\le j_0\frac{2sk}{l}\le2s(k-1).
\]
Thus there is an integer \(t\), \(1\le t\le2s\), such that
\[
2s+(t-1)(k-2)
\le j_0\frac{2sk}{l}
\le2s+t(k-2).
\]
Set
\[
m=(\beta)_t,\qquad
n=(\beta)_{\,t+j_0\,2sk/l}.
\]
The two elements lie in the same block of \(\Pi\). Moreover,
\(m=(t-1)k+1\), while the displayed inequalities imply that
\[
n\in\{(t-1)k+2,\ldots,tk\}.
\]
Hence \(m,n\) are distinct elements of the same cycle of \(\theta(e_1)\), contradicting \(\theta(e_1)\prec\Pi\). Therefore no such common block decomposition exists. In both the odd and even cases,
$\theta(e_1)=\beta^{-1}\gamma$, where
$\beta$ and $\gamma$ are $sd'$-cycles.
If $S_D$ is orientable, choose $v$ such that
$\gamma=v\beta v^{-1}$.
Taking
\[
\theta(a_1)=\beta,\qquad
\theta(b_1)=\beta^{-1}v
\]
and sending the remaining generators to the identity gives
\[
\theta(e_1)[\theta(a_1),\theta(b_1)]=1.
\]
If $S_D$ is nonorientable, choose $w$ such that
$\gamma=w\beta^{-1}w^{-1}$.
Taking
\[
\theta(a_1)=w,\qquad
\theta(a_2)=w^{-1}\beta
\]
and sending the remaining generators to the identity gives
\[
\theta(e_1)\theta(a_1)^2\theta(a_2)^2=1.
\]
In both cases, the generated group contains $\beta$ and
is therefore transitive. Lemmas \ref{lemma2.1} and \ref{lem:branched-cover} give the required
tiling, which is minimal by Theorem \ref{minimal criterion}.

We have proved the following result:

\begin{theorem} \label{onevertex}
Let $k\ge3$ be an integer, and suppose $D$ is a planar diagram with one vertex such that $\chi (S_D)\le 0$. Then for any integer $s\ge 1$, there exist minimal tilings with $sd'$ tiles by congruent most curvilinear polygons associated to $D$ with vertices $V^{k}$ on some surface.
\end{theorem}

The preceding argument does not apply when $k=2$. For example,
take $s=6$ and $l=3$. Then
\[
\theta(e_1)=(1,2)(3,4)\cdots(23,24),
\qquad
\beta=(1,3,5,\ldots,23,2,4,6,\ldots,24).
\]
The following partition is a block decomposition of order $3$:
\[
\begin{aligned}
\Pi={}&\{\{1,17,10\},\{3,19,12\},\{5,21,14\},\{7,23,16\},\\
    &{}\{9,2,18\},\{11,4,20\},\{13,6,22\},\{15,8,24\}\}.
\end{aligned}
\]
It satisfies $\theta(e_1)\prec\Pi$ and $\beta\in\Pi$.
Nevertheless, we have the following result:

\begin{theorem}\label{22}
For $k=2$, suppose $D$ is a planar diagram with one vertex such that $\chi (S_D)\le 0$. Then for any integer $s=2^m$ with an integer $m\ge0$, there exist minimal tilings by congruent most curvilinear polygons associated to $D$ with vertices $V^{k}$ and $sd'$ tiles on some surface.
\end{theorem}

\begin{proof}
In this case $d'=4$. The case $m=0$ follows from the existence and minimality
of a tiling with $d'$ tiles. Henceforth assume $m\ge1$. Again we may choose $\theta(e_1)=(1,2)(3,4)...(4s-1,4s)$, $\beta=(1,3,5,...,4s-1,2,4,6,...,4s)$. For any $l=2^n$, where $n\in \mathbb Z$ and $1\le n\le m$, we prove there is no $\Pi\in \Sigma_{4s}^l$ such that $\theta(e_1)\prec \Pi$ and $\beta\in \Pi$. Otherwise, the $i$-th entry and $(i+\frac{4s}{l})$-th entry will be in the same block of $\Pi$. Because $\frac{4s}{l}|2s$ in this case, $1$ and $2$ lie in the same block of $\Pi$. But $1$ and $2$ are in the same cycle of $\theta(e_1)$, so $\Pi$ can not be a decomposition associated to $\theta(e_1)$. By Lemma \ref{lemma3.1},
$\gamma=\beta\theta(e_1)$ is a $4s$-cycle.
The construction used in the proof of
Theorem \ref{onevertex} therefore gives a transitive
representation containing $\beta$.
The resulting tiling is minimal by
Theorem \ref{minimal criterion}.

\end{proof}

\section{Discussion: relation to Fuchsian groups}
\label{sec:fuchsian-groups}

The existence results above can also be interpreted in terms of
finite-index subgroups of the symmetry groups in Theorem~\ref{standard_symmetry}.
Let $D$ have vertices $V_1,\ldots,V_r$, assume $\chi(S_D)\leq0$, and
fix integers $k_1,\ldots,k_r\geq2$. Let $\widetilde{\mathcal{T}}$ be an associated
universal tiling with vertex types $V_1^{k_1},\ldots,V_r^{k_r}$, and
write $\Gamma=\operatorname{Sym}(\widetilde{\mathcal{T}})$ for the group given
in Theorem~\ref{standard_symmetry}. In particular, $\Gamma$ acts simply transitively on the tiles.

The quotient orbifold has underlying surface $S_D$ and cone points
of orders $k_1,\ldots,k_r$, so
\[
 \chi_{\mathrm{orb}}=\chi(S_D)-\sum_{i=1}^{r}\left(1-\frac1{k_i}\right)<0.
\]
After normalizing the curvature to $-1$, the universal cover is
$\mathbb H^2$. If $S_D$ is orientable, $\Gamma$ is a cocompact Fuchsian
group. If $S_D$ is nonorientable, $\Gamma$ is a cocompact non-Euclidean
crystallographic (NEC) group, whose orientation-preserving subgroup
$\Gamma^+$ is Fuchsian of index two. In both cases, a subgroup of
$\Gamma$ acts freely on $\mathbb H^2$ if and only if it is torsion-free:
point stabilizers in a discrete group are finite, and every
finite-order hyperbolic isometry has a fixed point.

We call a torsion-free subgroup $H\leq\Gamma$ \emph{maximal torsion-free}
if there is no torsion-free subgroup $K$ with $H\subsetneq K\leq\Gamma$.
This refers to maximality among torsion-free subgroups, rather than
maximality among all proper subgroups of $\Gamma$.

For a torsion-free subgroup $H\leq\Gamma$ of finite index, the quotient
$\mathcal{T}_H=\widetilde{\mathcal{T}}/H$ is a tiling of the closed hyperbolic surface
$\mathbb H^2/H$. Since $\Gamma$ acts simply transitively on the tiles,
\[
 \#\{\text{tiles of }\mathcal{T}_H\}=[\Gamma:H].
\]
The covering correspondence identifies a nontrivial descent of $\mathcal{T}_H$
with an inclusion $H\subsetneq K\leq\Gamma$ in which $K$ is torsion-free;
the degree of the descent is $[K:H]$. Consequently,
\[
 \mathcal{T}_H\text{ is minimal}
 \quad\Longleftrightarrow\quad
 H\text{ is maximal torsion-free in }\Gamma.
\]
Every surface tiling with universal tiling isometric to $\widetilde{\mathcal{T}}$
is obtained in this way, and isometric equivalence classes correspond
to conjugacy classes of these subgroups in $\Gamma$.

The following two propositions restate the existence results in group-theoretic
terms: 
\begin{proposition}[Finite-index torsion-free subgroups]
\label{prop:torsion-free-existence}
Let $\Gamma$ be the group of Theorem~\ref{standard_symmetry} with the data fixed above,
let $S$ be a closed connected surface, and let $d\geq1$ be an integer.
There exists a torsion-free subgroup $H\leq\Gamma$ with
$[\Gamma:H]=d$ and $\mathbb H^2/H$ homeomorphic to $S$ if and only if
\begin{enumerate}
\item $k_i\mid d$ for every $i$;
\item $\nu=\sum_{i=1}^{r}(d-d/k_i)$ is even;
\item $\chi(S)=d\chi(S_D)-\nu$;
\item if $S_D$ is orientable, then $S$ is orientable;
\item if $S_D$ is nonorientable and $S$ is orientable, then
$2k_i\mid d$ for every $i$.
\end{enumerate}
In particular, if the topological type of $S$ is not prescribed, the
possible indices are exactly the positive multiples of $d'$, where
$d'$ is the least positive integer such that $k_i\mid d'$ for every
$i$ and $\sum_{i=1}^{r}(d'-d'/k_i)$ is even.
\end{proposition}

This is the reformulation of Theorem~\ref{case12} in the present hyperbolic setting.
In contrast, minimality imposes maximality among torsion-free subgroups.
Theorems~\ref{theorem1}, \ref{onevertex}, and~\ref{22} give the following stronger existence statements.

\begin{proposition}[Maximal torsion-free subgroups of prescribed index]
\label{prop:maximal-torsion-free-existence}
With $\Gamma$ and $d'$ as above, there exists a maximal torsion-free
subgroup $H\leq\Gamma$ with $[\Gamma:H]=sd'$ in each of the following cases:
\begin{enumerate}
\item $r\geq2$ and $s$ is any sufficiently large prime;
\item $r=1$, $k_1\geq3$, and $s$ is any integer with $s\geq1$;
\item $r=1$, $k_1=2$, and $s=2^m$ for any integer $m\geq0$.
\end{enumerate}
For each fixed choice of the data in these cases, $\Gamma$ therefore
has infinitely many conjugacy classes of maximal torsion-free
subgroups of finite index.
\end{proposition}

The abstract group in Theorem~\ref{standard_symmetry} depends only on $D$ and the prescribed
vertex types. Hence these conclusions concern a fixed group even if
the geometric realizations of the prototile vary: group isomorphisms
preserve index, torsion-freeness, and maximality among torsion-free
subgroups. The assertion about infinitely many conjugacy classes
follows because the permitted indices are unbounded and conjugate
subgroups have the same index. The listed indices are not asserted
to exhaust all possible indices of maximal torsion-free subgroups.

One motivation for introducing minimal tilings was to seek a
classification of surface tilings by first classifying those that
cannot descend further. Every tiling descends to a minimal one after
finitely many nontrivial descents, since each descent strictly decreases
the positive integer counting its tiles. The results above show,
however, that even with $D$ and the vertex types fixed, minimal tilings
occur with arbitrarily many tiles. Thus minimality does not reduce
the classification problem to a finite list of tilings, nor does it
provide an upper bound on their numbers of tiles. In group-theoretic
terms, one must account for infinitely many conjugacy classes of
maximal torsion-free subgroups of the same group $\Gamma$.
This rules out a classification by finite enumeration under these
constraints, but does not rule out a classification by parameters,
algebraic invariants, or an effective procedure.
Here torsion-freeness is equivalent to a free action, not to being a
free group: $H$ is the fundamental group of a closed hyperbolic surface.
In the nonorientable case, maximality is taken in the full group
$\Gamma$, not merely in its Fuchsian subgroup $\Gamma^+$.

\end{document}